\documentclass[reqno,12pt]{amsart} 
\usepackage{vmargin}
\setpapersize{A4}

\usepackage{tikz}
\usepackage[latin1]{inputenc}
\usepackage{tikz,etoolbox}
\usepackage{amsmath} 

\usepackage{amsfonts}   

\usepackage{amssymb}      

\usepackage{eufrak}      

\usepackage{tikz}
\usetikzlibrary{matrix,calc}
\usepackage{tikz-cd}

\usepackage{amscd}      
\usepackage{amsthm}      

\usepackage{tikz,amsmath}
\usetikzlibrary{arrows}

\usepackage{epsfig}      

\usepackage{amstext}      

\usepackage{graphicx}

\usepackage[all,line,dvips]{xy} 
\CompileMatrices 

\newcommand{\Aut}{\operatorname{Aut}}

\newcommand{\Pro}{\operatorname{Pr}}

   \theoremstyle{plain}
   \newtheorem{thm}{Theorem}[section]
   \newtheorem{prop}[thm]{Proposition}
   \newtheorem{lemma}[thm]{Lemma}  
   \newtheorem{cor}[thm]{Corollary}
   \theoremstyle{definition}
   
   \newtheorem{defn}[thm]{Definition}
   \newtheorem{example}[thm]{Example}
   \theoremstyle{remark}
   
   \newtheorem{remark}[thm]{Remark}

\newcommand{\defeq}{:=}

\usepackage{mdframed,xcolor}

\definecolor{mybgcolor}{gray}{0.8}
\definecolor{myframecolor}{rgb}{.647,.129,.149}

\mdfdefinestyle{mystyle}{
  usetwoside=false,
  skipabove=0.6em plus 0.8em minus 0.2em,
  skipbelow=0.6em plus 0.8em minus 0.2em,
  innerleftmargin=.25em,
  innerrightmargin=0.25em,
  innertopmargin=0.25em,
  innerbottommargin=0.25em,
  leftmargin=-.75em,
  rightmargin=-0em,
  topline=false,
  rightline=false,
  bottomline=false,
  leftline=false,
  backgroundcolor=mybgcolor,
  splittopskip=0.75em,
  splitbottomskip=0.25em,
  innerleftmargin=0.5em,
  leftline=true,
  linecolor=myframecolor,
  linewidth=0.25em,
}

\newmdenv[style=mystyle]{important}

\usepackage{lipsum}

   \numberwithin{equation}{section}

        \date{\today}

\title[KMS weights on graph $C^*$-algebras]{KMS weights on graph $C^*$-algebras}  
\author{Klaus Thomsen}


\date{\today}

\email{matkt@math.au.dk}
\address{Institut for Matematik, Aarhus University, Ny Munkegade, 8000 Aarhus C, Denmark}

\begin{document}

\maketitle

\begin{abstract} The paper contains a study of the gauge invariant KMS
  weights for a generalized gauge action on a graph
  $C^*$-algebra. When the graph is irreducible and has the property
  that there are at most countably many roads to infinity in the graph,
  a complete description is given of the structure of KMS weights for the
  gauge action. The structure is very rich and is identical with the
  structure of KMS states for the restriction of the action to any corner defined by a projection in the
  fixed point algebra.
\end{abstract}

\section{Introduction} In 1980 Bratteli, Elliott and Kishimoto proved
a remarkable theorem concerning the structure of inverse
temperatures and the corresponding simplexes of KMS states for a
one-parameter group of automorphisms on a unital $C^*$-algebra,
\cite{BEK}. Their result says that if a
given structure of simplexes can be realized inside a metrizable
compact convex set in such a way that the closure properties which
the KMS
states have inside the state space of a unital $C^*$-algebra are satisfied, then the structure
is in fact the structure of KMS states of a periodic one-parameter
group of automorphisms acting on a unital and simple separable $C^*$-algebra. In
particular, it follows that any KMS structure which
occurs with some unital separable $C^*$-algebra, can also be realized with a
unital and simple separable
$C^*$-algebra. This fact is in striking contrast to the observation that in
practically all cases where it has been possible to determine the
structure of inverse temperatures and simplexes of KMS states of an a priori
given one-parameter action on a simple $C^*$-algebra, the
structure has been disappointingly poor; often with only one possible inverse
temperature and a unique KMS state. The gauge action on graph
$C^*$-algebras is no exception if one
sticks with finite graphs, \cite{EFW}, but it is the purpose of the present
paper to show how radically this changes when infinite graphs are considered.

We extend first the study of KMS weights on graph algebras which was
initiated in \cite{Th1} by allowing the graph to have sinks and
infinite emitters. In most of the paper we work in the same generality as in
\cite{Th1}, dealing with generalized gauge actions, but in this
introduction where only some of the results are described, attention
is restricted to the gauge action on the $C^*$-algebra of a strongly
connected graph $G$. There are then no sinks to consider, but there may be
plenty of infinite emitters. As in \cite{Th1} the KMS weights are
given by regular Borel measures on the path space of the graph, which
besides the infinite paths now also contains finite paths terminating
at an infinite emitter. This division of the path space leads to a
similar division of the KMS weights, depending on the supports of the
corresponding measures. If the measure is supported on the finite
paths, we say that the KMS weight is a \emph{boundary KMS weight} and if
the finite paths is a null set for the measure, that it is a
\emph{harmonic KMS weight}.

In order to have any KMS weights at all the adjacency matrix of the graph must have 'finite powers of all
orders'. This means that for a given vertex $v$
and a given natural number $n$
the number $a(n)$ of paths of length $n$ from $v$ back to itself must be
finite. In fact, the exponential growth rate of $a(n)$ must be
finite. The logarithm of this growth rate is the \emph{Gurevich
  entropy} $h(G)$ of the graph and there are no KMS weights when the Gurevich
entropy is infinite, and when it is finite there are no $\beta$-KMS
weights when $\beta < h(G)$. The graphs for which this paper describes the structure of
KMS weights completely are those with at most countably many exits. Here an \emph{exit} is a
tail-equivalence class of \emph{exit paths}, and an exit path is a
sequence $(t_i)_{i=1}^{\infty}$ of vertexes in the graph such that
there is an edge from $t_i$ to $t_{i+1}$ for all $i$ and such that
$t_i$ goes to infinity in the natural sense. Each exit
contributes an interval\footnote{Here and in the rest of the paper an
  interval is to be understood in the broadest sense possible. It may be empty,
open, closed, half-open or even degenerate, i.e. consist of a single
number only.} of inverse temperatures in $[h(G),\infty)$, and for each $\beta$ in the
interval there is an extremal ray of $\beta$-KMS weights, uniquely
determined by the condition that the corresponding measures are
supported on the exit. It is
these KMS weights that are responsible for the rich structure of the KMS
weights that may be realized with graphs of this kind, because it turns out that the intervals of inverse temperatures which the exits
contribute are independent and can be almost arbitrary. To formulate this
more precisely, we have to distinguish between when $G$ is
recurrent and when it is transient, \cite{Ru}. In terms of the numbers
$a(n)$ introduced above, $G$ is transient when the sum
$$
\sum_{n=1}^{\infty} a(n)e^{-nh(G)}
$$  
is finite and recurrent when it is not. Furthermore, we have to
distinguish between graphs that are row-finite, in the sense that the
out-degree at every vertex is finite, and those that are not
row-finite. Concerning the latter class of graphs we obtain the following theorems.

\begin{thm}\label{intro1} Let $N \in \{1,2,3, \cdots \} \cup
  \{\infty\}$ and let $h
  \in \big]0,\infty\big[$ be a positive real number. Let $\mathbb I$ be a finite or
  countably infinite collection of intervals in $\big]h,\infty\big[$. 

\smallskip

There is
  a strongly connected recurrent graph $G$ with Gurevich entropy
  $h(G) =h $, such that the set of exits in $G$ is in
  bijective correspondence with $\mathbb I$,
  and for $\beta \geq
  h$ there are the following extremal $\beta$-KMS weights for the
  gauge action on $C^*(G)$:
\begin{enumerate}
\item[$\bullet$] For $\beta > h$ there are $N$ extremal rays of boundary $\beta$-KMS weights,
  in bijective correspondence with the infinite emitters in $G$, and
  the rays of extremal harmonic $\beta$-KMS weights are in bijective correspondence with the set
$$
\left\{ I \in \mathbb I :  \ \beta \in I \right\} .
$$
\item[$\bullet$] For $\beta = h$ there are no boundary $h$-KMS weights
  and a unique ray of extremal harmonic $h$-KMS weights.
\end{enumerate}
\end{thm}

\bigskip

\begin{thm}\label{intro} Let $N \in \{1,2,3,\cdots \} \cup \{\infty\}$
  and let $h
  \in \big]0,\infty\big[$ be a positive real number. Let $\mathbb I$ be a finite or
  countably infinite collection of intervals in $\big[h,\infty\big[$. 

\smallskip

There is
  a strongly connected transient graph $G$ with Gurevich entropy $h(G)
  = h$, such that the set of exits in $G$ is in 
  bijective correspondence with $ \mathbb I$, 
  and for $\beta \geq
  h$ there are the following rays of extremal $\beta$-KMS weights for the
  gauge action on $C^*(G)$: There are $N$ extremal rays of boundary $\beta$-KMS weights,
  in bijective correspondence with the infinite emitters in $G$, and
  the rays of extremal harmonic $\beta$-KMS weights are in bijective
  correspondence with the set $\left\{ I \in \mathbb I :  \ \beta \in I \right\}$.
\end{thm}

Before we get to the construction of the graphs mentioned in the two theorems, we obtain results demonstrating
that the structures described are the most general that can be obtained from
strongly connected graphs with infinite emitters and at most countably
many exits.

The row-finite case, where there are no infinite emitters, must be
handled separately because the results from \cite{Th1} show that there
are $\beta$-KMS weights for all $\beta > h(G)$ when $G$ is not finite and hence the
total freedom in the choice of intervals in the theorems above does
not persist to the row-finite case. We show that in the row-finite case there must
be at least one exit which contributes the maximal possible interval
of inverse temperatures, namely $]h(G),\infty[$ in the recurrent
case and $[h(G),\infty[$ in the transient case. Once this restriction
is identified it is not difficult to modify the construction from the
case with infinite emitters and show that it is also the only restriction. In
this way we obtain results, formulated in
Theorem \ref{REV1} and Theorem \ref{REV2} below, which describe the
possibilities in the row-finite case.

The proofs of the main results are constructive in a very literal sense. They present a small collection of building blocks and describe
methods to choose and put together a finite or countable
collection of such building blocks to form a strongly connected graph
for which the gauge action has any desired KMS spectrum, restricted
only by the general limitations which graphs with at most countably
many exits are subject to.

By restricting attention to corners in $C^*(G)$ the above
structure of KMS weights becomes the structure of KMS states on a
unital and simple $C^*$-algebra: A vertex $v$ in $G$ determines a projection $1_v$ in $C^*(G)$ in a canonical
way, and the gauge action restricts to the corner $1_vC^*(G)1_v$ in
$C^*(G)$. There is then a bijective correspondence between the rays of KMS
weights on $C^*(G)$ and the KMS states on $1_vC^*(G)1_v$, and the
results of the paper describe therefore the structure of KMS states that can
occur on such a corner when the graph has at
most countably many exits.

\smallskip

\emph{Acknowledgment:} I am grateful to Johannes Christensen for his help with
  the revision of this paper, and to the Mittag-Leffler Institute for
  support during a part of the work.

\section{KMS weights, measures and almost harmonic vectors}

Recall, \cite{KV}, \cite{Th1}, that a weight $\psi$ on
the $C^*$-algebra $A$ is \emph{proper} when it is non-zero, densely defined and lower
semi-continuous. For such a weight, set
 $$
\mathcal N_{\psi} = \left\{ a \in A: \ \psi(a^*a) < \infty
\right\}.
$$ 
Let $\alpha : \mathbb R \to \Aut A$ be a point-wise
norm-continuous one-parameter group of automorphisms on
$A$. Let $\beta \in \mathbb R$. Following \cite{C} we say that a proper weight
$\psi$ on $A$ is a \emph{$\beta$-KMS
  weight} for $\alpha$ when
\begin{enumerate}
\item[i)] $\psi \circ \alpha_t = \psi$ for all $t \in \mathbb R$, and
\item[ii)] for every pair $a,b \in \mathcal N_{\psi} \cap \mathcal
  N_{\psi}^*$ there is a continuous and bounded function $F$ defined on
  the closed strip $D_{\beta}$ in $\mathbb C$ consisting of the numbers $z \in \mathbb C$
  whose imaginary part lies between $0$ and $\beta$, and is
  holomorphic in the interior of the strip and satisfies that
$$
F(t) = \psi(a\alpha_t(b)), \ F(t+i\beta) = \psi(\alpha_t(b)a)
$$
for all $t \in \mathbb R$. \footnote{As in \cite{Th1} we apply the
  definition from \cite{C} for the action $\alpha_{-t}$
  in order to use the same sign convention as in \cite{BR}, for example.}
\end{enumerate}   
A $\beta$-KMS weight $\psi$ with the property that 
$$
\sup \left\{ \psi(a) : \ 0 \leq a \leq 1 \right\} = 1
$$
will be called a \emph{$\beta$-KMS state}. This is consistent with the
standard
definition of KMS states, \cite{BR}, except when $\beta = 0$ in which
case our definition requires also that a $0$-KMS state, which is a
trace state, is $\alpha$-invariant. A $\beta$-KMS weight is \emph{extremal} when every $\beta$-KMS weight
$\varphi$ such that $\varphi \leq \psi$ has the form $\varphi =
\lambda \psi$ for some $\lambda \in ]0,1]$.

\subsection{The \'etale groupoid of a countable graph}

Let $G$ be a countable directed graph with vertex set $V$ and edge set
$E$. For an edge $e \in E$ we denote by $s(e) \in V$ its source and by
$r(e) \in V$ its range. An \emph{infinite path} in $G$ is an element
$p  \in E^{\mathbb N}$ such that $r(p_i) = s(p_{i+1})$ for all
$i$.  A finite path $p = p_1p_2 \dots p_n $ is
defined similarly. The number of edges in $p$ is its \emph{length}
and we denote it by $|p|$. A vertex $v \in V$ will be considered as
a finite path of length $0$.

We let $P(G)$ denote the set of
infinite paths in $G$ and $P_f(G)$ the set of finite paths in $G$. We
extend the source map to $P(G)$ such that
$s(p) = s(p_1)$ when $p = \left(p_i\right)_{i=1}^{\infty}$, and the
range and source maps to $P_f(G)$ such that $s(p) = s(p_1)$ and $r(p)
= r(p_n)$ when $|p| = n \geq 1$, and $s(v) = r(v) = v$ when $v\in V$.

A vertex $v$ which does not emit any edge is
a \emph{sink}, while a vertex $v$ which emits infinitely many edges will be called an \emph{infinite
  emitter}. The union $V_{\infty}$ of sinks and infinite emitters will
play a crucial role in the following.

The $C^*$-algebra $C^*(G)$ of the graph $G$ is the universal
$C^*$-algebra generated by a collection $S_e, e \in E$, of partial
isometries and a collection $P_v, v \in V$, of mutually orthogonal projections subject
to the conditions that
\begin{enumerate}
\item[1)] $S^*_eS_e = P_{r(e)}, \ \forall e \in E$,
\item[2)] $S_eS_e^* \leq P_{s(e),} \ \forall e \in E$,
\item[3)] $P_v \geq \sum_{e \in s^{-1}(v)} S_eS_e^*, \ \forall v \in
  V$, and
\item[4)] $P_v = \sum_{e  \in s^{-1}(v)} S_eS_e^*, \ \forall v \in V
  \backslash V_{\infty}$.
\end{enumerate}

It will be crucial for our approach to the graph $C^*$-algebra
$C^*(G)$ that it can be realized as the (reduced) $C^*$-algebra
$C^*_r(\mathcal G)$ of an \'etale groupoid
$\mathcal G$ through the construction introduced by J. Renault in
\cite{Re}. The relevant groupoid $\mathcal G$ was constructed by A. Paterson in \cite{Pa}. See in
particular Corollary 3.9 in \cite{Pa}.

As a set the unit space $\Omega_G$ of the groupoid $\mathcal G$ is the union
$$
\Omega_G = P(G) \cup Q(G), 
$$
where 
$$
Q(G) = \left\{p \in P_f(G): \ r(p) \in V_{\infty} \right\} 
$$ 
is the set of finite paths that terminate at a vertex
in $V_{\infty}$. Note that $Q(G)$ is countable and that $V_{\infty} \subseteq Q(G)$. For any $p \in P_f(G), |p| \geq 1$, set
$$
Z(p) = \left\{ q \in \Omega_G: \ |q| \geq |p| , \ q_i = p_i, \ i = 1,2,
  \cdots, |p| \right\},
$$
and
$$
Z(v) = \left\{ q \in \Omega_G : \ s(q) = v\right\}
$$
when $v \in V$. When $\nu \in P_f(G)$ and $F$ is a finite subset of $P_f(G)$, set
\begin{equation}\label{a6}
Z_F(\nu) = Z(\nu) \backslash \left(\bigcup_{\mu \in F} Z(\mu)\right) .
\end{equation}
The sets $Z_F(\nu)$ form a basis of compact and open subsets for a locally compact Hausdorff
topology on $\Omega_G$. When $\mu \in P_f(G)$ and $  x \in \Omega_G$, we can define the
concatenation $\mu x \in \Omega_G $ in the obvious way when $r(\mu) =
s(x)$. The groupoid $\mathcal G$ consists of the
elements in $\Omega_G \times \mathbb Z \times \Omega_G$ of the form
$$
(\mu x, |\mu| - |\mu'|, \mu'x),
$$
for some $x\in \Omega_G$ and some $ \mu,\mu' \in P_f(G)$. The product
in $\mathcal G$ is defined by
$$
(\mu x, |\mu| - |\mu'|, \mu' x)(\nu y, |\nu| -|\nu'|, \nu' y) = (\mu
x, \ |\mu | + |\nu| - |\mu'| - |\nu'|, \nu' y),
$$ 
when $\mu' x = \nu y$, and the involution by $(\mu x, |\mu| - |\mu'|,
\mu'x)^{-1} = (\mu' x, |\mu'| - |\mu|, \mu x)$. To describe the
topology on $\mathcal G$, let $Z_{F}(\mu)$ and $Z_{F'}(\mu')$ be two
sets of the form (\ref{a6}) with $r(\mu) = r(\mu')$. The topology we
shall consider has as a countable basis the sets of the form
\begin{equation}\label{top}
\left\{ (\mu x, |\mu| - |\mu'|, \mu' x) : \ \mu x \in Z_F(\mu), \
  \mu'x \in Z_{F'}(\mu') \right\} .
\end{equation}
With this topology $\mathcal G$ becomes an \'etale locally compact Hausdorff groupoid and we can consider the reduced $C^*$-algebra $C^*_r(\mathcal
G)$ as in
\cite{Re}. As shown by Paterson in
\cite{Pa} there is an isomorphism $C^*(G) \to
  C^*_r(\mathcal G)$ which sends $S_e$ to $1_e$, where $1_e$ is the
  characteristic function of the compact and open set
$$
\left\{ (ex, 1, r(e)x) : \ x \in \Omega_G \right\} \ \subseteq \
\mathcal G,
$$  
and $P_v$ to $1_v$, where $1_v$ is the characteristic function of the
compact and open set
$$
\left\{ (vx,0,vx) \ : \ x \in \Omega_G \right\} \ \subseteq \ \mathcal G.
$$ 
In the following we use the identification $C^*(G) = C_r^*(\mathcal
G)$ and identify $\Omega_G$ with the unit space of $\mathcal G$ via
the embedding
$\Omega_G \ni x \ \mapsto \ ( x, 0,x)$.

\begin{remark}\label{a45} The set $P(G)$ is usually considered as a
  metric space, e.g. with the metric
$$
d(p,q) = \sum_{i=1}^{\infty} 2^{-i} (1-\delta(p_i,q_i))
$$
where $\delta(e,e) = 1$ and $\delta(e,f) = 0$ when $e \neq f$. It is easy
to see that the topology on $P(G)$ defined by such a metric is the
same as the topology which $P(G)$ inherits as a subset of $\Omega_G$. In
particular, the two topologies define the same Borel subsets in $P(G)$.   
\end{remark}

\subsection{Generalized gauge actions on $C^*(G)$ and their
  gauge invariant KMS weights}

Let $F : E \to \mathbb R$ be a function. We extend $F$ to a function
$F : P_f(G) \to \mathbb R$ such that 
$$
F(p_1p_2\cdots p_n) = \sum_{i=1}^n F(p_i)
$$
when $p = p_1p_2 \cdots p_n$ is a path of length $n \geq 1$ in $G$, and
$F(v) = 0$ when $v \in V$. We can
then define a continuous function $c_{F} : \mathcal G \to \mathbb R$
such that
$$
c_{F}(ux,|u|-|u'|,u'x) = F(u)- F(u').
$$
Since $c_{F}$ is a continuous homomorphism it gives rise to a continuous
one-parameter automorphism group $\alpha^F$ on $C^*_r(\mathcal G)$
defined such that
$$
\alpha^{F}_t(f)(\gamma) = e^{it c_{F}(\gamma)} f(\gamma)
$$
when $f \in C_c(\mathcal G)$, cf. \cite{Re}. When $F$ is constant
$1$ this action is known as \emph{the gauge action} on $C^*(G)$.

Let $\beta \in \mathbb R$. Following the terminology used in \cite{Th1} we say that a regular Borel measure $m$ on
$\Omega_G$ is \emph{$(\mathcal G, c_{F})$-conformal with exponent $\beta$} when
\begin{equation}\label{c31} 
m(s(W)) = \int_{r(W)} e^{\beta c_{F}(r_W^{-1}(x))} \ dm(x)
\end{equation}
for every open bi-section $W \subseteq \mathcal G$. Here $r_W^{-1}$
denotes the inverse of $r: W \to r(W)$. 
When the function $F$ is fixed we shall often in the following refer to a $(\mathcal G,
c_{F})$-conformal measure with exponent $\beta$ as a \emph{$\beta$-KMS
  measure.} The connection to $\beta$-KMS weights is given by the following theorem which is a special case of
Theorem 2.2 in \cite{Th1}.

\begin{thm}\label{a1} There is a bijective correspondence $m \mapsto
  \varphi_m$ between the non-zero $(\mathcal G, c_{F})$-conformal
  measures $m$ with exponent $\beta$ and the gauge invariant $\beta$-KMS weights
$\varphi_m$ for the action $\alpha^{F}$
on $C^*(G)$. The bijection is such that
$$
\varphi_m(f) = \int_{\Omega_G} f(z) \ dm(z)
$$
when $f \in C_c(\mathcal G)$.
\end{thm}

In terms of the canonical generators, the $\beta$-KMS weight
$\varphi_m$ is given by
$$
\varphi_m\left(S_eS_f^*\right) = \delta(e,f) e^{-\beta F(e)} m\left(Z(r(e))\right) \
\text{and} \ \varphi_m(P_v) = m(Z(v)).
$$

\begin{remark}\label{a44} To obtain the theorem above we apply Theorem 2.2
  in \cite{Th1} with $c = c_F$ and $c_0$ equal
  to the homomorphism that gives the gauge action, i.e. $c_0 = c_F$
  with $F = 1$. However, when $F$ is strictly positive everywhere
  or strictly negative everywhere we can apply Theorem 2.2 in
  \cite{Th1} with $c = c_0 = c_F$ instead and
  obtain in that case a version of Theorem \ref{a1} with the words
  'gauge invariant' deleted. This shows that when $F$ is either
  strictly positive or strictly negative everywhere, all KMS
  weights for $\alpha^F$ are gauge invariant. The same is
  true whenever $C^*(G)$ is simple by Proposition 5.6 in \cite{CT2}. In
  all these cases the results we obtain about gauge invariant KMS
  weights therefore hold with the words 'gauge invariant' deleted. But
  in general when zero is a not a forbidden value for $F$, or $F$ is
  allowed to change sign, there can be KMS weights and states that are
  not gauge invariant. See
  \cite{N} and \cite{CT1}. 
\end{remark}

Before we restrict the attention entirely to KMS weights rather than
states, we want to point
out that when $C^*(G)$ is simple there is a bijection from rays of KMS weights on $C^*(G)$
onto the KMS states of some of its corners. By a \emph{ray} of KMS
weights we mean here a set of the form $\left\{ \lambda \psi : \ \lambda >
  0 \right\}$ for some KMS weight $\psi$.

\begin{thm}\label{h5}  Let $\alpha : \mathbb R \to \Aut A$ be a point-wise
norm-continuous one-parameter group of automorphisms on a $C^*$-algebra
$A$. Let $p$ be a projection in the fixed point algebra of
$\alpha$ such that $p$ is full in $A$. For all $\beta \in \mathbb R$ the map
$$
\psi \mapsto  \psi(p)^{-1} \psi|_{pAp}
$$
is a bijection between the set of rays of $\beta$-KMS weights for $\alpha$ and
the $\beta$-KMS states for the restriction of $\alpha$ to $pAp$.
\end{thm}
\begin{proof} It follows from Theorem 3.2 in \cite{LN} that a $\beta$-KMS state for
  the restriction of $\alpha$ to $pAp$ has a unique extension to a
  $\beta$-KMS weight on $A$, cf. Remark 3.3 (i) in \cite{LN}. Hence
  all we need to do is to prove that the map in the statement is
  well-defined, and for this  it suffices to show that $0 < \psi(p) <
  \infty$ when $\psi$ is a $\beta$-KMS weight for $\alpha$. It was shown in Lemma 3.1 of
  \cite{CT2} that $\psi(p)< \infty$. Assume for a
  contradiction that $\psi(p) = 0$. It follows from Proposition 2.5.22
  in \cite{BR}
  that $A$ contains a dense $*$-algebra $\mathcal A$ consisting of
  analytic elements for $\alpha$.  Let $a,b \in \mathcal A$ and set
  $a' = \alpha_{\frac{i\beta}{2}}(a)$.
Then $ap = \alpha_{-\frac{i \beta}{2}}(a'p)$ and hence
$$
\psi(apbb^*pa^*) \leq  \|b\|^2 \psi(apa^*)  = \|b\|^2 \psi\left( \alpha_{-\frac{i
      \beta}{2}}(a'p) \alpha_{-\frac{i
      \beta}{2}}(a'p)^*\right) .
$$
By using the alternative formulation of the KMS condition given in Proposition 1.11 in \cite{KV}, we find that
$$
\psi\left( \alpha_{-\frac{i
      \beta}{2}}(a'p) \alpha_{-\frac{i
      \beta}{2}}(a'p)^*\right) = \psi(p{a'}^*a'p) \leq \|a'\|^2
\psi(p) = 0.
$$
Thus $\psi(apbb^*pa^*) = 0$. By taking $b =
p$ it follows in particular that $pa^* \in \mathcal N_{\psi}$. If we instead take $a =
b^*$ and $b =p$, it follows that $pb \in \mathcal N_{\psi}$, and we conclude that $apb =(pa^*)^*pb
\in \mathcal N_{\psi}^*\mathcal N_{\psi}$. Now recall that $\psi$ extends to a
positive linear functional on the $*$-algebra $\mathcal N_{\psi}^*\mathcal
N_{\psi}$, cf. page 842 in \cite{KV}. We can therefore use the Cauchy-Schwarz 
inequality to conclude that $\psi(xx^*) = 0$ for all $x$ that are
linear combinations of elements $apb$ with $a,b \in \mathcal A$. That
$p$ is full means that elements of the
form $apb$ span a dense subspace of
$C^*(G)$, and from what we have just shown it follows that $\psi(xx^*) =
0$ for all elements $x$ in this subspace. The lower semi-continuity of
$\psi$ implies then that $\psi = 0$; a contradiction. Hence $\psi(p) >
0$ as required.
\end{proof} 

For the actions on $C^*(G)$ we consider here, a natural choice of
projection to which Theorem \ref{h5} applies is one of the projections $P_v$. This leads to the following

\begin{cor}\label{h775} Let $C^*(G)$ be a simple graph $C^*$-algebra
  and $v$ a vertex in $G$. Let $P_v$ be the corresponding projection. The map 
$$
\psi \mapsto
  \psi(P_v)^{-1} \psi|_{P_vC^*(G)P_v}
$$ 
is a bijection from
  the rays of $\beta$-KMS weights for $\alpha^F$ on $C^*(G)$ onto the $\beta$-KMS states for the restriction of $\alpha^F$
  to $P_vC^*(G)P_v$.
\end{cor}

\subsection{KMS measures and super-harmonic functions}

Given the function $F : E \to \mathbb R$ and a real number $\beta \in
\mathbb R$ we define the matrix $A(\beta) =
(A(\beta)_{uw})$ over $V$ by
$$
A(\beta)_{uw} \ = \sum_{ \left\{ e \in E : \ s(e) = u, \ r(e) = w\right\}}
e^{-\beta F(e)} .
$$
When $\beta = 0$ the matrix $A = A(0)$ is the \emph{adjacency
  matrix} of $G$, i.e.
\begin{equation}\label{h210}
A_{vw} \ = \ \# \left\{ e \in E: \ s(e) =v, \ r(e) =w \right\} .
\end{equation}
Note that $A(\beta)_{uw}$ can be infinite, i.e. $A(\beta)_{uw} \in
[0,\infty]$. Nonetheless we can define the powers $A(\beta)^n$ of
$A(\beta)$ in the
usual recursive way:
$$
A(\beta)^n_{uw} = \sum_{v \in V} \ A(\beta)_{uv}  A(\beta)^{n-1}_{vw} ,
$$
where we use the convention that $0 \cdot \infty = \infty \cdot 0 =
0$. We define $A(\beta)^0$ to be the identity matrix, i.e. $A(\beta)^0_{uw} = 1$ when $u
= w$ and $A(\beta)^0_{uw} = 0$ when $u\neq w$. We shall use the matrix
$A(\beta)$ to count the weighted paths between vertexes in $G$. For
this note that for each $n \in \mathbb N$, 
$$
A(\beta)^n_{vw} = \sum_{\mu} e^{-\beta F(\mu)}
$$
where the sum is over all paths $\mu$ of length $n$ from $v$ to $w$.

Given a $\beta$-KMS measure $m$ on $\Omega_G$ we define $\psi_v, \ v \in
V$, such that
\begin{equation}\label{b100}
\psi_v = m(Z(v)).
\end{equation}
Note that $\psi_v < \infty$ since $m$ is regular and $Z(v)$ is compact
in $\Omega_G$.

\begin{lemma}\label{b5} Let $m$ be a $\beta$-KMS measure on
  $\Omega_G$. The vector 
$\psi_v = m(Z(v)), \ v \in V$,
has the following two properties:
\begin{enumerate}
\item[1)] $\sum_{w \in V}  A(\beta)_{vw}\psi_w \leq \psi_v, \ v \in V$, and
\item[2)] $\sum_{w \in V}  A(\beta)_{vw}\psi_w = \psi_v, \ v \in V \backslash
  V_{\infty}$.
\end{enumerate}
The identity 2) holds for all $v \in V$ when $m(Q(G)) = 0$. 
\end{lemma}
\begin{proof} Consider a vertex $v \in V$ and an edge $e\in
  s^{-1}(v)$. Then
$\left\{ (ex,1,x) : \ x \in Z(r(e))\right\}$
is an open bi-section in $\mathcal G$. As $m$ is $(\mathcal
G,c_F)$-conformal with exponent $\beta$ this implies that
\begin{equation*}\label{n200}
m(Z(r(e)) = e^{\beta F(e)} m(Z(e)) .
\end{equation*}
Assume first that $v \in V_{\infty}$. Then $Z(v)$ is a disjoint union $Z(v)
= \{v\} \cup \bigcup_{e \in s^{-1}(v)} Z(e)$, and hence
\begin{equation*}
\begin{split}
&\psi_v =  m(\{v\})  + \sum_{e \in s^{-1}(v)}
m(Z(e)) = m(\{v\})  + \sum_{e \in s^{-1}(v)}
e^{-\beta F(e)} m(Z(r(e))) \\
&= m(\{v\}) + \sum_{w \in V} A(\beta)_{vw} m(Z(w)) \ \geq \  \sum_{w \in V} A(\beta)_{vw} \psi_w.
\end{split}
\end{equation*} 
This shows that 1) holds when $v \in V_{\infty}$.  When $v \in V
\backslash V_{\infty}$ or $m(\{v\}) = 0$ the term $m(\{v\})$ does not enter and we
obtain 2) instead.
\end{proof}

It follows from 1) in Lemma \ref{b5} by induction that
\begin{equation}\label{b7}
\sum_{w \in V} A(\beta)^n_{vw}\psi_w \leq \psi_v
\end{equation}
for all $n \in \mathbb N$ and all $v\in V$. In the following 
we shall say that $\psi$ is
\emph{$A(\beta)$-harmonic} when
$$
 \sum_{w \in V}  A(\beta)_{vw}\psi_w = \psi_v
$$
for all $v \in V$, and that $\psi$ is \emph{almost
  $A(\beta)$-harmonic} when conditions 1) and 2) in Lemma \ref{b5} both hold. This terminology is
inspired by the notion of harmonic and super-harmonic functions used in
the theory of Markov chains, cf. \cite{Wo}.

We aim to prove the following theorem.

\begin{thm}\label{almost} The map $m \mapsto \psi$ given by
  (\ref{b100}) is a bijection from $\beta$-KMS measures on $\Omega_G$
  onto the set of almost
  $A(\beta)$-harmonic vectors on $V$.
\end{thm}

The injectivity of the map $m \mapsto \psi$ is a consequence of the
following lemmas.

\begin{lemma}\label{b101} Let $m$ be a $\beta$-KMS measure on
  $\Omega_G$. Then
\begin{equation}\label{b102}
m(Z(\mu)) = e^{-\beta F(\mu)} m\left(Z\left(r(\mu)\right)\right)
\end{equation}
for all $\mu \in P_f(G)$.
\end{lemma}
\begin{proof} This follows from (\ref{c31}) applied to the open
  bisection
$$
W =
\left\{ (\mu x, |\mu|, r(\mu)x)  : \ x \in Z(r(\mu)) \right\} .
$$

\end{proof}

\begin{lemma}\label{twice} Let $v \in V$ be a vertex and let
  $m$, $m'$ be 
  finite Borel measures on $Z(v)$. Assume that $m(Z(\mu)) =
  m'(Z(\mu))$ for every finite path $\mu \in P_f(G)$ with $s(\mu) =
  v$. Then $m = m'$.
\end{lemma}
\begin{proof} The intersection of a finite collection of cylinder sets
  $Z(\mu)$ is again a cylinder set or empty. These cylinder sets
  therefore form a $\pi$-system in the sense of \cite{Co} and the
  conclusion follows from Corollary 1.6.2 in \cite{Co} since
  the $\sigma$-algebra generated by the cylinder sets contains all open sets and
  hence the Borel $\sigma$-algebra.
\end{proof}

\emph{Proof of injectivity in Theorem \ref{almost}:} Let $m$ and $m'$
be $\beta$-KMS measures on $\Omega_G$ such that $m\left(Z(v)\right) =
m'\left(Z(v)\right)$ for every vertex $v \in V$. Fix a vertex $v \in
V$ and note that $m(Z(v)) = m'(Z(v)) < \infty$ by regularity of the
two measures. It follows from Lemma \ref{b101} that
$m\left(Z(\mu)\right) = m'\left(Z(\mu)\right)$ and then from Lemma
\ref{twice} that $m$ and $m'$ agree on all Borel subsets of $Z(v)$. Since $v \in V$ was arbitrary and $Z(v), v \in V$,
is a countable Borel partition of $\Omega_G$ it follows that the
two measures are identical. \qed

\smallskip

For the proof of the surjectivity part in Theorem \ref{almost} we
shall use
the following lemma which is an extension of the Riesz
decomposition used in connection with countable state Markov chains,
e.g. Lemma 4.2 in \cite{Sa}. The only difference is that we here do
not require that the matrix is stochastic or sub-stochastic, but the
proof is the same. In the present setting the details of the proof
are presented in \cite{Th3}.

\begin{lemma}\label{b1} 
Let $\psi = (\psi_v)_{v \in V} \in [0,\infty[^V$ be a non-negative
vector such that 
\begin{equation*}\label{b2}
\sum_{w\in V} A(\beta)_{vw} \psi_w \leq \psi_v
\end{equation*}
 for all $v\in V$. It follows that there are unique non-negative
 vectors $h, k \in [0,\infty[^V$ such that $h$ is $A(\beta)$-harmonic
 and
\begin{equation}\label{b3}
\psi_v = h_v + \sum_{w \in V} \sum_{n=0}^{\infty} A(\beta)^n_{vw}k_w 
\end{equation}
for all $v \in V$. The vector $k$ is given by
$$
k_v= \psi_v - \sum_{w \in V} A(\beta)_{vw}\psi_w, \ v \in V .
$$ 
\end{lemma}

To prove the surjectivity of the map in Theorem
\ref{almost} we shall consider the
two components in the decomposition (\ref{b3}) of an almost
$A(\beta)$-harmonic vector separately. One virtue of this approach is that it
shows how the decomposition of an almost $A(\beta)$-harmonic vector
given by Lemma \ref{b1} corresponds to the decomposition of $\Omega_G$ as
the disjoint union of $P(G)$ and $Q(G)$.

\section{Boundary and harmonic KMS measures}

Recall that a subset $A \subseteq \Omega_G$ is \emph{$\mathcal
  G$-invariant} when $r\left(s^{-1}(A)\right) \subseteq A$ and $s\left(r^{-1}(A)\right) \subseteq A$.

\begin{lemma}\label{h6} Let $m$ be $(\mathcal G, c_F)$-conformal with
  exponent $\beta$ and
  let $A \subseteq \Omega_G$ be a $\mathcal G$-invariant Borel
  subset. Then the Borel measure $m_A$, given by
$$
m_A(B) = m(A \cap B),
$$
is $(\mathcal G, c_F)$-conformal with exponent $\beta$.
\end{lemma}
\begin{proof}Let $W$ be an open bi-section in $\mathcal G$. Since $m$
  is $(\mathcal G,c_F)$-conformal with exponent $\beta$, the two Borel
  measures on $W$,
$$
B \mapsto m(s(B))
$$ 
and
$$
B \mapsto  \int_{r(B)} e^{\beta c_{F}(r_W^{-1}(x))} \ dm(x),
$$
agree on open sets. Note that they are both regular by Proposition
7.2.3 in \cite{Co} and therefore equal since they agree on open sets. It
follows that
$$
m_A\left(s(W)\right) = m\left(s(W \cap s^{-1}(A))\right) =
\int_{r\left(W \cap s^{-1}(A)\right)} e^{\beta
  c_F\left(r_W^{-1}(x)\right)} \ dm(x).
$$
Since $r(W \cap s^{-1}(A)) = r(W) \cap A$ because $A$ is
$\mathcal G$-invariant,
\begin{equation*}
\begin{split}
&\int_{r\left(W \cap s^{-1}(A)\right)} e^{\beta
  c_F\left(r_W^{-1}(x)\right)} \ dm(x) \\
&= \int_{r(W) \cap A} e^{\beta
  c_F\left(r_W^{-1}(x)\right)} \ dm(x) = \int_{r\left(W \right)} e^{\beta
  c_F\left(r_W^{-1}(x)\right)} \ dm_A(x).
\end{split}
\end{equation*}
\end{proof}

\begin{lemma}\label{a14} Let $m$ be a $\beta$-KMS measure on $\Omega_G$. There are
  unique $\beta$-KMS measures $m_1$ and $m_2$ such that $m = m_1+
  m_2$, and $m_1(Q(G)) = m_2(P(G)) = 0$.   
\end{lemma}
\begin{proof} This follows from Lemma \ref{h6} since $P(G)$ and $Q(G)$
  are $\mathcal G$-invariant in $\Omega_G$.
\end{proof}

The decomposition in Lemma \ref{a14} is a version for weights of the
decomposition of KMS states into finite and infinite type used by
Carlsen and Larsen in \cite{CL}. Since 'finite' and 'infinite' have
other meanings in connection with measures and weights we prefer to
alter the terminology. We will say that a $\beta$-KMS measure $m$ is a \emph{boundary
  $\beta$-KMS
measure} when $m(P(G)) = 0$ and a
\emph{harmonic $\beta$-KMS measure} when $m\left(Q(G)\right) =
0$. 
This terminology is justified (or so the author hopes) by the fact
that $\Omega_G$ in many cases is a completion of $P(G)$ with
$Q(G)$ as the boundary, and by the description of the harmonic KMS
measures we obtain in the following. 

From now on we will tacitly
identify a harmonic $\beta$-KMS measure with its restriction to
$P(G)$. The $\beta$-KMS weight defined by a non-zero boundary
$\beta$-KMS measure or a non-zero harmonic
$\beta$-KMS measure will be called a \emph{boundary $\beta$-KMS
  weight} and a \emph{harmonic $\beta$-KMS weight}, respectively.

\subsection{Boundary KMS measures}\label{bounKMS}

Let $\beta \in \mathbb R$. For any vertex $v \in V_{\infty}$ we can consider
the Borel measure 
$$
m_v \ = \sum_{\left\{ u
  \in Q(G) : \ r(u) = v\right\}} \ e^{-\beta F(u)} \delta_u
$$
on $\Omega_G$ where $\delta_u$ denotes the Dirac measure at
$u$. Consider an open bisection $W$ in
$\mathcal G$. Then $T = r\circ s^{-1} : s(W)\cap Q(G) \to r(W)\cap Q(G)$ is a bijection
determined by the condition that $(T(u), |T(u)|-|u|, u) \in W$. Since
$c_F\left( T(u), |T(u)| -|u|,u\right) = F(T(u)) - F(u)$ we find that
\begin{equation*}
\begin{split}
&m_v(s(W)) \ = \sum_{\left\{u \in s(W) \cap Q(G) : \ r(u) =v
  \right\} } \ e^{-\beta F(u)} \\
& \\
&= \sum_{\left\{u \in s(W) \cap Q(G) : \ r(u) =v
  \right\} } \ e^{\beta (F(T(u))-F(u))} e^{-\beta F(T(u))} \\
& \\
&=   \sum_{\left\{u \in s(W) \cap Q(G) : \ r(u) =v
  \right\} } \ e^{\beta \left(c_F\left(r_W^{-1}(T(u))\right)\right) }
e^{-\beta F(T(u))} =
 \int_{r(W)} e^{\beta c_F\left(r_W^{-1}(x)\right)} \ d m_v(x) .
\end{split}
\end{equation*}
This shows that $m_v$ satisfies condition (\ref{c31}), and hence is a
$\beta$-KMS measure if and only if it is regular. We say that a vertex $v \in V_{\infty}$ is \emph{$\beta$-summable} when 
$$
\sum_{n=0}^{\infty} A(\beta)^n_{wv}   \ < \ \infty 
$$
for all $w \in V$. 

\begin{lemma}\label{b104} Let $v \in V_{\infty}$. The Borel measure $m_v$ is regular, and
  hence a $\beta$-KMS measure if and only if $v$ is $\beta$-summable.
\end{lemma}
\begin{proof} $m_v$ is regular if and only $m_v(K) < \infty$ for every
  compact subset $K \subseteq \Omega_G$, cf. e.g. Proposition 7.2.3 in
  \cite{Co}. Since $Z(w), \ w\in V$, is a cover of $\Omega_G$ by open
  and compact subsets it follows that $m_v$ is regular if and only if
  $m_v(Z(w)) < \infty$ for all $w \in V$. The lemma follows therefore
  from the observation that
\begin{equation}\label{b124}
m_v(Z(w)) \ \ =  \sum_{\left\{ u
  \in Q(G) : \ s(u) = w, \ r(u) = v\right\}}  \ e^{-\beta
F(u)} \ \ = \ \ \sum_{n=0}^{\infty} A(\beta)^n_{wv} .
\end{equation}
\end{proof}

\begin{thm}\label{a16} Let $S$ be a set of $\beta$-summable vertexes
  in $V_{\infty}$ and $t \in ]0,\infty[^S$ a vector such that
\begin{equation}\label{a19}
\sum_{v \in S} \sum_{n=0}^{\infty} A(\beta)^n_{wv} t_v < \infty 
\end{equation}
for all $w \in V$. Then 
\begin{equation}\label{b120}
m = \sum_{ v \in S} t_v m_v
\end{equation}
is a boundary $\beta$-KMS measure, and every non-zero boundary $\beta$-KMS measure
arises in this way.
\end{thm}
\begin{proof} Note that $m$ satisfies (\ref{c31}) since each $m_v$
  does. The condition (\ref{a19}) is therefore exactly what is needed
  to ensure that $m$ is regular and hence a boundary
  $\beta$-KMS measure. It remains therefore only
  to prove that every non-zero boundary $\beta$-KMS measure $m$ arises like
  this. Consider an element $u \in Q(G)$ and set $v = r(u) \in V_{\infty}$. Assume first that $v$ is
  an infinite emitter. Let $F_1 \subseteq F_2
  \subseteq F_3 \subseteq \cdots $ be finite subsets of edges such that
  $s^{-1}(v) = \bigcup_i F_i$. For each $i$
$$
W_i =\left\{ (ux,|u|, vx): \ x \in Z_{F_i}(v) \right\}
$$
is an open bisection in $\mathcal G$ such that $s(W_i) =
Z_{F_i}(v)$ and $r(W_i) = Z_{uF_i}(u)$, where
$$
uF_i = \left\{ ux: \ x \in F_i \right\}. 
$$
Note that $\bigcap_i Z_{F_i}(v) = \{v\}$, $\bigcap_i
Z_{uF_i}(u) = \{u\}$ and $c_F(ux,|u|,vx) = F(u)$. It follows therefore
from (\ref{c31}) that
\begin{equation}
\begin{split} 
&m(\{v\}) = \lim_{ i \to \infty} m(Z_{F_i}(v)) =  \lim_{i \to
  \infty} \int_{r(W_{i})} e^{\beta c_F\left(r_{W_i}^{-1}(z)\right)} \
dm(z)\\
& = \lim_{i \to
  \infty} \int_{r(W_{i})} e^{\beta F(u)} \ dm(z)
= e^{\beta F(u)} m(\{u\}).
\end{split}
\end{equation}
If instead $v$ is a sink, the point $(u,|u|,v)$ is isolated in
$\mathcal G$, and it is an open bisection in itself. It follows
therefore from (\ref{c31})
that 
$$
m(\{v\}) = e^{\beta F(u)} m(\{u\}), 
$$
also when $v$ is a sink.
For each $w \in V$, 
\begin{equation}\label{b110}
\begin{split}
& \sum_{n=0}^{\infty} A(\beta)^n_{wv} m(\{v\}) \ \ =  \sum_{\left\{ u
  \in Q(G) \cap Z(w) : \ r(u) = v\right\}} \ e^{-\beta F(u)}
m(\{v\})\\
& =  \sum_{\left\{ u
  \in Q(G) \cap Z(w) : \ r(u) = v\right\}} m(\{u\}) \ 
\leq \ m(Z(w)) \ < \ \infty,
\end{split}
\end{equation}
by regularity of $m$. Hence $v$ is $\beta$-summable whenever $m(\{v\})
> 0$.
Set $S = \left\{ v \in V_{\infty}  : \ m (\{v\}) \neq 0 \right\}$. By
assumption $m$ is supported on the countable set $Q(G)$ and hence
\begin{equation*}
\begin{split}
m = \sum_{u \in Q(G)} m(\{u\}) \delta_u = \sum_{v \in S}   \sum_{\{u  \in Q(G): \ r(u) = v\}} m(\{v\})
e^{-\beta F(u)} \delta_u = \sum_{v \in S}
m(\{v\}) m_v.
\end{split}
\end{equation*}
This shows that (\ref{b120}) holds when we set $t_v = m(\{v\})$. To show that also (\ref{a19}) holds, note that the sets $\left\{ u
  \in Q(G) \cap Z(w) : \ r(u) = v\right\}, v \in S$, are mutually
disjoint subsets of $Z(w)$. Using (\ref{b110}) this implies that
$$
\sum_{v \in S} \sum_{n=0}^{\infty} A(\beta)^n_{wv} m(\{v\}) =  \sum_{v
  \in S} \sum_{\left\{ u
  \in Q(G) \cap Z(w) : \ r(u) = v\right\}} m(\{u\}) \leq m(Z(w)) < \infty
.
$$
\end{proof}

 The decomposition (\ref{b120}) is unique since $m_v(\{v'\}) = 0$ when
 $v' \neq v, \ v,v' \in V_{\infty}$. Theorem \ref{a16} therefore has the following

\begin{cor}\label{b121} The map $v \mapsto m_v$ gives a bijective
  correspondence from the set of $\beta$-summable vertexes in
  $V_{\infty}$ onto the rays of non-zero extremal boundary $\beta$-KMS measures.
\end{cor}

The corresponding result for states is Corollary 5.18(1) and
Proposition 5.8 in \cite{CL}.

\subsection{ Harmonic KMS measures}\label{harm}

Let $m$ be a harmonic $\beta$-KMS measure on $\Omega_G$ and consider
the vector $\psi$ given by (\ref{b100}). It follows from Lemma \ref{b5} that $\psi$ is $A(\beta)$-harmonic. To show that all
$A(\beta)$-harmonic vectors arise this way it is convenient to relate
our setup to the theory of countable state Markov chains through the
following device.

Let $\psi$ be a non-zero $A(\beta)$-harmonic vector. Set
$$
E_0 = \left\{ e \in E: \ \psi_{r(e)} \neq 0 \right\} .
$$
We define a matrix $B(\beta)$ over $E_0$ by
\begin{equation}\label{Bbeta}
B(\beta)_{e_0,e_1} = \begin{cases} \psi_{r(e_0)}^{-1} \psi_{r(e_1)} e^{-\beta
  F(e_1)}, &  \ \text{when} \ r(e_0) = s(e_1), \\ 0, & \ \text{otherwise.} \end{cases}
\end{equation}
Then 
$$
\sum_{e_1 \in E_0} B(\beta)_{e_0,e_1} = \psi_{r(e_0)}^{-1} \sum_{e_1 \in E_0,
  \ s(e_1) = r(e_0)} e^{-\beta F(e_1)} \psi_{r(e_1)} =
\psi_{r(e_0)}^{-1} \sum_{w \in V}A(\beta)_{r(e_0),w} \psi_w  = 1
$$
since $\psi$ is $A(\beta)$-harmonic. Thus $B(\beta)$ is stochastic
and gives therefore rise to a Markov chain with $E_0$ as state space
for any choice of initial distribution on $E_0$,
cf. Theorem 1.12 (b) in \cite{Wo}. When we fix an edge $e \in E_0$ and
take as the initial distribution the Dirac measure at $e$ we get in
this way a probability measure $\text{Pr}_e$ on $E_0^{\mathbb N}$ defined on the
$\sigma$-algebra $\mathcal A$ generated by the cylinder sets in $E_0^{\mathbb N}$ by
Theorem 1.12 (a) in \cite{Wo}. When $(a_i)_{i=0}^n \in E_0^{n+1}$, the
corresponding cylinder set in $E_0^{\mathbb N}$ is
\begin{equation}\label{cyl23}
C\left((a_i)_{i=0}^n \right) =  \left\{ (x_i)_{i=0}^{\infty} \in E_0^{\mathbb N} : x_i = a_i, i =
  0,1,2, \cdots ,n \right\}.
\end{equation}
By definition
\begin{equation}\label{cyl24}
\text{Pr}_e\left( C\left((a_i)_{i=0}^n \right)\right) =  \psi_{r(e)}^{-1}
\psi_{r(a_n)} e^{-\beta \sum_{i=1}^nF(a_i) }
\end{equation}
when $a_0 = e$ and $r(a_i) = s(a_{i+1}),i = 0,1,2,\cdots, n-1$, while
$\text{Pr}_e\left(C\left((a_i)_{i=0}^n \right)\right) = 0$ in all other cases. 

In the following we set $Z'(\mu) = Z(\mu) \cap P(G)$ when $\mu \in P_f(G)$.

\begin{lemma}\label{boet} $B \cap E_0^{\mathbb N} \in \mathcal A$ for every Borel subset $B \subseteq Z'(e)$, and 
\begin{equation}\label{measure}
B \mapsto \Pro_e(B \cap  E_0^{\mathbb N})
\end{equation}
is a Borel probability measure
on $Z'(e)$.
\end{lemma}
\begin{proof} The subsets $A$ of $Z'(e)$ such that $A \cap
  E_0^{\mathbb N} \in \mathcal A$ form a $\sigma$-algebra and the
  cylinders $Z'(\mu)$ generate the Borel
  $\sigma$-algebra in $Z'(e)$ so it suffices to show that
  $Z'(\mu) \cap E_0^{\mathbb N} \in \mathcal A$ when $\mu$ is
  a finite path in $G$ with $e$ as initial edge. Fix such a path
  $\mu$. When $\nu = e_0e_1\cdots e_n \in P_f(G)$ we set
$$
C(\nu) = \left\{ (x_i)_{i=0}^{\infty} \in E_0^{\mathbb N} : \ x_i =
  e_i, \ i =0,1,\cdots, n\right\}.
$$
Then $C(\nu) = \emptyset$ when $e_i \notin E_0$ for some $i$, and
otherwise $C(\nu)$ is the cylinder set $C\left((e_i)_{i=0}^n
\right)$. In any case $C(\nu) \in \mathcal A$. Let $P_n(\mu)$ denote the set of finite paths $\nu$
in $G$ of length $n \geq |\mu|$ such that $Z(\nu) \subseteq Z(\mu)$. Then
$$
Z'(\mu) \cap E_0^{\mathbb N} = \bigcap_{k > |\mu|}
\left(\bigcup_{\nu \in P_k(\mu)} C(\nu)\right) \in \mathcal A,
$$
completing the proof of the first assertion. For the second note that
$E_0^{\mathbb N} \backslash Z'(e)$ is contained in a
countable union of cylinder sets (\ref{cyl23}) for each of which either $a_0 \neq e$ or $r(a_i) \neq s(a_{i+1})$ for some $i =
0,1,2,\cdots, n-1$. Such sets are null-sets for $\text{Pr}_e$ by
definition. Hence $\text{Pr}_e\left(Z'(e)\right) =
\text{Pr}_e\left(E_0^{\mathbb N}\right) =1$.
\end{proof}

It can happen that $Z'(e) \nsubseteq E_0^{\mathbb N}$, but it follows
from Lemma \ref{boet} that $\text{Pr}_e(E_0^{\mathbb N} \backslash
Z'(e)) = 0$ which justifies that we in the following use the notation
$\text{Pr}_e$ also for the measure on $Z'(e)$ given by (\ref{measure}).

\begin{lemma}\label{a20} Let $\psi$ be an $A(\beta)$-harmonic vector. There is a harmonic $\beta$-KMS measure $m_{\psi}$ on $\Omega_G$ such
that $\psi_v = m_{\psi}\left(Z(v)\right)$ for all $v \in V$.
\end{lemma}
\begin{proof} When $\psi = 0$, set $m_{\psi} =0$. Assume then that
  $\psi \neq 0$. Let $e \in E_0$ and let $\text{Pr}_e$ be the Borel probability measure on $Z'(e)$ from Lemma \ref{boet}. For every vertex $v \in V$ we define a
  Borel measure $m^v$ on $Z(v)$ such that $m^v =0$ when $s^{-1}(v)
  \cap E_0 = \emptyset$ and 
$$
m^v(B) = \sum_{e \in s^{-1}(v) \cap E_0} e^{-\beta F(e)} \psi_{r(e)} \text{Pr}_e\left(B \cap
Z'(e)\right) ,
$$
when $ s^{-1}(v)
  \cap E_0 \neq \emptyset$. Then 
$$
m^v(Z(v)) = \sum_{e \in s^{-1}(v)
    \cap E_0} e^{-\beta F(e)} \psi_{r(e)} = \sum_{w \in
    V}A(\beta)_{vw} \psi_w = \psi_v.
$$ 
When we define $m_{\psi}$ such that 
$$
m_{\psi}(B) = \sum_{v \in
  V}m^v(B \cap Z(v))
$$
we have therefore obtained a Borel measure on $\Omega_G$ such that
$m_{\psi}(Z(v)) = m_{\psi}(Z'(v)) = \psi_v$ for all $v \in V$. In
particular, $m_{\psi}$ is finite on compact subsets of $\Omega_G$
because such a set is covered by finitely many $Z(v)$'s. Hence
$m_{\psi}$ is
a regular Borel measure by Proposition 7.2.3 in \cite{Co}. Note that $m_{\psi}$ is
supported on $P(G)$. What remains is now only to show that (\ref{c31}) holds when $W$ is an open bisection
in $\mathcal G$. For this note first that from the definition of $m_{\psi}$
and (\ref{cyl24}) it follows that
\begin{equation}\label{cyl25}
m_{\psi}(Z(\mu)) = m_{\psi}(Z'(\mu)) = e^{-\beta F(\mu)} \psi_{r(\mu)}
\end{equation}
for all $\mu \in P_f(G)$. Next observe that $W$ is the union of
countably many sets of the form (\ref{top}). Note that the set
(\ref{top}) is a Borel subset of $\left\{(\mu x, |\mu| -|\mu'|,\mu'x):
  \ x \in \Omega_G\right\}$. We can therefore write $W$ as a countable
disjoint union $W = \bigcup_i W_i$ of Borel subsets $W_i$ such that
there are finite paths $\mu_i, \mu_i' \in P_f(G)$ with $r(\mu_i) =
r(\mu_i')$ and a subset $B_i \subseteq \Omega_G$ with
$$
W_i = \left\{ (\mu_i x, |\mu_i| - |\mu_i'|, \mu_i'x) : \ x \in B_i \right\} .
$$
Note that $B_i$ is a Borel subset of $\Omega_G$ because $r : W \to
r(W)$ is a homeomorphism. Since the $W_i$'s are mutually disjoint, to
establish (\ref{c31}) it suffices to fix $i$ and show that   
\begin{equation}\label{enough!}
m_{\psi}(s(W_i)) = \int_{r(W_i)}  e^{\beta c_F\left( r_W^{-1}(z)\right)} \
dm_{\psi}(z).
\end{equation}
Note that $s(W_i) = \left\{\mu_i'x : \ x \in B_i\right\}$ and $r(W_i) = \left\{\mu_ix : \ x \in B_i\right\}$. Let $v = r(\mu_i) = r(\mu'_i)$ and consider the two finite Borel measures on $Z(v)$ given by
$$
B \mapsto \int_{\left\{ \mu_i x : \ x \in
    B \right\} } e^{\beta c_F\left( r_W^{-1}(z)\right)} \
dm_{\psi}(z)
$$
and
$$ 
B \mapsto m_{\psi} \left( \left\{ \mu_i' x : \ x \in
    B \right\}\right) .
$$
To show that they are
equal it suffices therefore by Lemma \ref{twice} to check that they
agree on  cylinder sets
$Z(\mu)$ with $s(\mu) = v$. For this note that $c_F$ is constant equal to
$F(\mu_i) - F(\mu'_i)$ on $W_i$. By using
(\ref{cyl25}) it follows that
\begin{equation}
\begin{split}
& \int_{\left\{ \mu_i x : \ x \in
    Z(\mu) \right\} } e^{\beta c_F\left( r_W^{-1}(z)\right)} \
dm_{\psi}(z) = e^{\beta\left(F(\mu_i) - F(\mu'_i)\right)} m_{\psi}\left( \left\{ \mu_i x : \ x \in
    Z'(\mu) \right\}\right) \\
& \\
& = e^{\beta\left(F(\mu_i) - F(\mu'_i)\right)} e^{-\beta F(\mu_i\mu)}
\psi_{r(\mu)} =  e^{-\beta F(\mu'_i\mu)} 
\psi_{r(\mu)} = 
m_{\psi} \left( \left\{ \mu_i' x : \ x \in
    Z'(\mu) \right\}\right) .
\end{split}
\end{equation}
Thus $m_{\psi}$ is a harmonic $\beta$-KMS measure and the proof is complete.

\end{proof}

\begin{thm}\label{a17} The map $m \mapsto \psi$ given by (\ref{b100})
  is a bijection from the
  harmonic $\beta$-KMS measures $m$ on $\Omega_G$ onto the $A(\beta)$-harmonic
  vectors for $A(\beta)$.
\end{thm}
\begin{proof} Surjectivity is Lemma
  \ref{a20} and injectivity was proved after Lemma \ref{twice}.
  \end{proof}

\emph{Proof of surjectivity in Theorem \ref{almost}:} Let $\psi$ be an
almost $A(\beta)$-harmonic vector. We must find a $\beta$-KMS measure
$m$ such that $m(Z(u)) = \psi_u$ for all $u \in V$. To this end let
$h,k$ be the vectors arising from Lemma \ref{b1}. Note that $k$ is
supported on $V_{\infty}$ since $\psi$ is almost
$A(\beta)$-harmonic. It follows that
$$
\sum_{v \in V_{\infty}} \sum_{n=0}^{\infty} A(\beta)^n_{wv}k_v =
\psi_w - h_w \leq \psi_w < \infty .
$$
and then $m_2 = \sum_{v \in V_{\infty}} k_v m_v$ is a boundary
$\beta$-KMS measure by Theorem \ref{a16}. It follows from Theorem
\ref{a17} that there is a harmonic $\beta$-KMS measure $m_1$ such that
$m_1(Z(v)) = h_v$. Then $m = m_1 + m_2$ is a $\beta$-KMS
measure, and by combining (\ref{b124}) with (\ref{b3}) we find that 
$$
m(Z(u)) = h_u \ + \sum_{v \in V_{\infty}} k_v m_v(Z(u)) = h_u \ +
\sum_{v \in V_{\infty}} \sum_{n=0}^{\infty} A(\beta)^n_{uv} k_v
= \psi_u . \qed
$$

To sum up, by combining Theorem \ref{a1} and Theorem \ref{almost} we
get the following corollary which will be the basis for the following
investigations.

\begin{cor}\label{summary} There are affine bijections between the following
  three sets.
\begin{enumerate}
\item[$\bullet$] The gauge invariant $\beta$-KMS weights for
  $\alpha^F$.
\item[$\bullet$] The non-zero $\beta$-KMS measures on $\Omega_G$.
\item[$\bullet$] The non-zero almost $A(\beta)$-harmonic vectors on
  $V$.
\end{enumerate}
\end{cor}

Equally important is the division of the KMS weights which arise from
these bijections into harmonic KMS weights and boundary KMS weights.

\section{KMS weights on simple graph $C^*$-algebras}\label{KMSsimple}

Recall that a subset $H \subseteq V$ is \emph{hereditary} when $e \in
E, \ s(e) \in H \Rightarrow r(e) \in H$ and \emph{saturated} when 
$$
v \in V\backslash V_{\infty}, \  
r(s^{-1}(v))  \subseteq H \ \Rightarrow \ v\in H.
$$ 
In the following we say that $G$ is \emph{cofinal} when the only
non-empty subset of $V$ which is both hereditary and
saturated is $V$ itself. This condition is fulfilled when $C^*(G)$ is
simple, and it is a result of Szymanski that the converse is
almost also true, cf. Theorem 12 in \cite{Sz}. We note that when $G$
is \emph{strongly connected}, meaning that for every pair of vertexes $v,w$ there is
a finite path $\mu$ such that $s(\mu ) = v$ and $r(\mu) =w$, then $V$
contains no proper non-empty hereditary subset and $G$ is therefore also cofinal.

\begin{lemma}\label{n17} Let $G$ be cofinal and $H \subseteq V$ a non-empty hereditary
  subset. Set $H_0 = H$ and define $H_i, i \geq 1$, such that
$$
H_i = H_{i-1} \cup \left\{v \in V \backslash V_{\infty} : \
  r(s^{-1}(v)) \subseteq H_{i-1} \right\} .
$$
Then $\bigcup_{i=0}^{\infty} H_i = V$.
\end{lemma}
\begin{proof} The union is hereditary and saturated.
\end{proof}

Since a sink is a hereditary subset of $V$, Lemma \ref{n17} has the following

\begin{cor}\label{b126} Let $G$ be a cofinal graph. Then
\begin{enumerate}
\item[1)]  $G$ contains at most one sink. 
\item[2)]  If $P(G) \neq \emptyset$ there is no sink in $G$.
\item[3)] If $V$ contains a sink, there is for each $u \in V$ a
  natural number $N$ such that $A(\beta)^n_{uw} = 0$ for all $w \in V$
  when $n \geq N$.
\end{enumerate}
\end{cor}

As in \cite{Th1} we consider the set $NW_G$ of non-wandering vertexes.
These are the vertexes $v$ for which there is a finite
path (a loop) $\mu \in P_f(G)$ such that $|\mu| \geq 1$ and $s(\mu) = r(\mu) =
v$. When $G$ is cofinal the vertexes in $NW_G$ and the edges $e\in E$ with $s(e) \in NW_G$ form a
sub-graph of $G$ which we shall also denote by $NW_G$.

\begin{prop}\label{NWG} Assume that $G$ is cofinal and that $NW_G \neq
  \emptyset$. The sub-graph $NW_G$ is strongly connected and
  the vertexes in $NW_G$ form a hereditary subset of
  $V$ with the following property: For every $v \in V$ there is an
  $N_v \in \mathbb N$
  such that $r(\nu) \in NW_G$ for all $\nu\in P_f(G)$ with $|\nu|\geq
  N_v$ and $s(\nu) = v$.
\end{prop}
 \begin{proof} Consider a path $e_1e_2 \cdots e_n \in P_f(G)$ such
   that $n \geq 1$ and $s(e_1) = r(e_n)$, and let $L$ be the set of
   vertexes $v$ with the property that $v =s (\mu)$ for some $\mu \in
   P_f(G)$ with $r(\mu) \in \bigcup_{i=1}^n
  \{r(e_i)\}$. Then $V \backslash L$ is both hereditary and
  saturated, and it is not all of $V$ since $s(e_1) \in L$, so it is
  empty. This has several consequence. It follows in particular that we can go back and forth
  between loops in $G$ and hence that $NW_G$ is strongly
  connected. It also implies that $NW_G$ is hereditary. For this it
  suffices to show that $r(e) \in
  NW_G$ when $s(e) = s(e_1)$. But since $r(e) \in L$ there is a path
  from $r(e)$ to $s(e)$ and $r(e)$ is therefore contained in a loop,
  i.e. $r(e) \in NW_G$. Finally, to establish the last statement we
  apply Lemma \ref{n17} with $H = NW_G$. It follows that any vertex
  $v$ is in $H_{N_v}$ for some $N_v \in \mathbb N$. By definition of the
  $H_i$'s this implies that any path of
  length $\geq N_v$ starting at $v$ must end in $NW_G$.
\end{proof}

Proposition \ref{NWG} shows that $NW_G$ works as a sort of black hole
in $G$ when $G$ is cofinal.

\begin{lemma}\label{a80}  Assume that $G$ is cofinal. No vertex $v \in V \backslash NW_G$ is an
  infinite emitter.
\end{lemma}
\begin{proof} Let $v \in V$ be an infinite emitter. Set
$$
A = \left\{w \in V: \ w = s(\mu) \ \text{for some $\mu \in P_f(G)$ such
    that} \ r(\mu) = v \right\} .
$$ 
Then $v \in A$, and since $V \backslash A$ is hereditary and saturated, it follows
that $A = V$. In particular, $r(s^{-1}(v)) \subseteq A$, which implies
that $v\in NW_G$.    
\end{proof}

In the following we
will say that \emph{all powers of $A(\beta)$ are finite} when
$A(\beta)^n_{uw} < \infty$ for all $n \in \mathbb N$ and all $u,w \in V$.

\begin{lemma}\label{a39} Assume that $G$ is cofinal. Let $\psi \in
  [0,\infty[^V$ be a $A(\beta)$-harmonic vector. Assume that $\psi$ is
  not identical zero. Then $\psi_v > 0$ for all $v \in
  V$ and all powers of $A(\beta)$ are finite.
\end{lemma}
\begin{proof} That $\psi$ must be strictly positive follows from the
  observation that 
$$
\left\{v \in V : \ \psi_v = 0 \right\}
$$ 
is
  hereditary and saturated. It follows then from (\ref{b7}) that all
powers of $A(\beta)$ are finite.
\end{proof}

\begin{lemma}\label{a41}  Assume that $G$ is cofinal. Let $\beta \in
  \mathbb R$. 
\begin{enumerate}
\item[1)] A sink $v \in V_{\infty}$ is $\beta$-summable.
  \item[2)] An infinite emitter $v \in V_{\infty}$ is $\beta$-summable
  if and only if $\sum_{n=0}^{\infty} A(\beta)^n_{vv} < \infty$. This
  condition implies that all powers of $A(\beta)$ are finite.
\end{enumerate} 
\end{lemma}
\begin{proof} 1) Let $w \in V$. It follows from Corollary \ref{b126} that
  there is an $N_w \in \mathbb N$ such that $\sum_{n=0}^{\infty} A(\beta)^n_{wv} =
\sum_{n=0}^{N_w} A(\beta)^n_{wv}$. This sum is finite since $V = V
\backslash NW_G$ by 2) of Corollary \ref{b126} and $V\backslash NW_G$ contains no infinite emitter by Lemma \ref{a80}.

2) The set
\begin{equation}\label{n18}
\left\{ u \in V : \ \sum_{n=0}^{\infty} A(\beta)^n_{uv} < \infty
\right\}
\end{equation}
is hereditary and saturated. Cofinality of $G$ therefore implies the
first statement. To
  prove the second statement, consider two vertexes $w,u \in V$ and a
$k \in \mathbb N$. It follows from Lemma \ref{a80} that $v \in
NW_G$. In particular, $NW_G$ is not empty and hence 2) of Corollary
\ref{b126} implies that there are no sinks. By Proposition \ref{NWG} $NW_G$ is a strongly connected subgraph of $G$ whose
vertexes constitute a hereditary subset of $V$, and it follows
therefore from Lemma
\ref{n17} that there is an $l \in \mathbb N$ such that
$A(\beta)^l_{uv} \neq 0$. Since
$$
A(\beta)^k_{wu}A(\beta)^l_{uv} \leq A(\beta)^{k+l}_{wv} < \infty,
$$
it follows that $A(\beta)^k_{wu} < \infty$.
\end{proof}

\begin{cor}\label{b128} Assume that $G$ is cofinal. There are no
  non-zero $\beta$-KMS measures unless all powers of $A(\beta)$ are finite.
\end{cor}
\begin{proof} It follows from Lemma \ref{a39} and Theorem \ref{a17}
  that there are no non-zero harmonic $\beta$-KMS measures unless all
  powers of $A(\beta)$ finite. It follows from 2) of Lemma \ref{a41}
  that there are no non-zero boundary $\beta$-KMS measures unless all
  powers of $A(\beta)$ are finite.
 \end{proof}

In view of Corollary \ref{b128} we shall in the following implicitly assume that
all powers of $A(\beta)$ are finite.

\subsection{No non-wandering vertexes}\label{nowan}

We split now the considerations into three cases, depending of the
size of $NW_G$. We begin with the case where $NW_G = \emptyset$. Then $G$ has no
infinite emitters by Lemma \ref{a80} and at most one sink by 1) in Lemma
\ref{b126}. In particular, $G$ is row-finite and except for the
possible presence of a sink, the case is covered by Corollary 7.3 in \cite{Th2}.

\begin{thm}\label{h1} Assume that $G$ is cofinal and that $NW_G$ is
  empty. 

\begin{enumerate}
\item[a)] Assume that $G$ contains a sink. For every $\beta \in \mathbb
  R$ there is a $\beta$-KMS weight for $\alpha^F$. It is unique up multiplication by
  scalars, and is given by the boundary $\beta$-KMS measure of
  the sink.
\item[b)] Assume that there is no sink in $G$. For any $\beta \in
  \mathbb R$ there are $\beta$-KMS weights, and they
  are all harmonic and gauge invariant.
\end{enumerate}
\end{thm}
\begin{proof} $C^*(G)$ is simple in this case by Theorem 12 in
  \cite{Sz} and it follows from Proposition 5.6 in \cite{CT2} that all
  KMS weights are gauge invariant. a): It follows from Corollary \ref{b126} that there is only
  one sink and that $P(G) = \emptyset$. In particular, there are no
  non-zero harmonic $\beta$-KMS measures. There are no infinite
  emitters in $G$ by Lemma \ref{a80} and the sink is therefore the
  only element in $V_{\infty}$. Since the sink is
  $\beta$-summable by 1) of Lemma \ref{a41}, the statements follow
  from Theorems \ref{a1}, \ref{almost} and \ref{a16}. b): This case is covered by Corollary 7.3 in \cite{Th2}.
\end{proof}

In general, the $\beta$-KMS weights in case b) are not unique, not
even up to multiplication by scalars. In fact, work of Kishimoto, \cite{Ki},
suggests that the KMS spectrum can be very rich, but we will not
pursue this case further in the present work.

\subsection{A reduction} We want to study the KMS weights of $\alpha^F$
when $G$ is cofinal and $NW_G$ not empty, and in this section we show
that for this purpose we can assume that $G$ is strongly connected.

Let $G$ be a cofinal graph and assume that
$NW_G \neq \emptyset$. Then $NW_G$ is a strongly connected subgraph of
$G$ by Proposition \ref{NWG}. Note that
$$ 
U = \left\{ x \in \Omega_G : \ s(x) \in NW_G\right\}
$$
is a closed and open subset of $\Omega_G$ which we can and will identify
with $\Omega_{NW_G}$. The reduction 
$$
\mathcal G|_{U} = \left\{ \xi \in \mathcal G: \ s(\xi) \in U,
  \ r(\xi) \in U \right\}
$$
is an \'etale locally compact second countable Hausdorff groupoid and the corresponding
convolution $C^*$-algebra $C^*_r\left(\mathcal G|_{U}\right)$ is
isomorphic to $C^*(NW_G)$. In this way we get an embedding
$$
C^*(NW_G) \subseteq C^*(G)
$$ 
of $C^*(NW_G)$ as a hereditary $C^*$-subalgebra of $C^*(G)$,
cf. Proposition 1.9 in \cite{Ph}. Note that $\alpha^F$ leaves $C^*(NW_G)$ globally invariant.

\begin{prop}\label{reduction} Assume that $G$ is cofinal and that
  $NW_G \neq \emptyset$. For each $\beta \in \mathbb R$ the restriction $\psi \mapsto
  \psi|_{C^*(NW_G)}$ is a bijection from the $\beta$-KMS weights for $\alpha^F$ on $C^*(G)$ to those for the
  restriction of $\alpha^F$ to $C^*(NW_G)$. 
\end{prop}
\begin{proof} In view of Theorem \ref{h5} it suffices to take a vertex
  $v \in NW_G$ and show that $P_v$ is full in both $C^*(G)$ and
  $C^*(NW_G)$. Let $I$ be an ideal in $C^*(G)$
  containing $P_v$. Then $\left\{ w \in V: \ P_w \in I\right\}$ is not empty. If $e
  \in E$ and $P_{s(e)} \in I$ we find that $P_{r(e)} = S_e^*S_e \in
  I$, and if $v \in V \backslash V_{\infty}$ and $P_{r(e)} \in I$ for
  all $e \in s^{-1}(v)$ we find that $P_v = \sum_{e \in  s^{-1}(v)}
  S_eS_e^* = \sum_{e \in s^{-1}(v)} S_eP_{r(e)}S_e^* \in I$. It
  follows then from the
  cofinality of $G$ that $\left\{ w \in V: \ P_w \in I\right\} = V$
  and hence that $I = C^*(G)$. Thus $P_v$ is full in $C^*(G)$. Since $NW_G$ is strongly connected an
  even simpler argument shows $P_v$ is also full in $C^*(NW_G)$. 
\end{proof}

In view of Proposition \ref{reduction} and since we have handled the
case when $NW_G = \emptyset$ in Section \ref{nowan}, we can restrict attention to
$C^*(NW_G)$. Thus in the sequel we shall assume that $G$ is strongly
connected. But then $C^*(G)$ is simple by Theorem 12 in \cite{Sz},
except when $G$ consists of a single loop and nothing more. In that case
all KMS weights for $\alpha^F$ can be normalized to states and they
can be
found and described by use of results from \cite{CT1}. There is
therefore no loss in assuming that $G$ is strongly connected and that $C^*(G)$
is simple. This will simplify many formulations in the following
because all KMS
weights for $\alpha^F$ are then gauge invariant by Proposition 5.6 in
\cite{CT2}.

\subsection{A Hopf dichotomy} In \cite{CL} Carlsen and Larsen
introduced an interesting division of the harmonic KMS states on a
general graph $C^*$-algebra. It is this division we study in this section, but for
KMS weights on strongly connected graphs. The goal
is to show that harmonic $\beta$-KMS measures are always either
dissipative or conservative, in a sense we now make precise. We assume
throughout that $G$ is strongly connected.
Set 
\begin{equation}\label{rec7}
P(G)_{rec} = \bigcap_{e \in E} \left\{ (x_i)_{i =1}^{\infty} \in P(G): \ x_i = e
   \ \text{for infinitely many} \ i  \right\} ,
\end{equation} 
and 
\begin{equation}\label{wan7}
P(G)_{wan} = \bigcap_{v \in V} \left\{ (x_i)_{i=1}^{\infty} \in
  P(G) : \ \# \left\{ i \in \mathbb N: \ s(x_i) = v\right\} < \infty  \right\} .
\end{equation}
A Borel measure $m$ on $P(G)$ is \emph{conservative} when
$m(P(G) \backslash P(G)_{rec}) = 0$, and \emph{dissipative} when
$m(P(G) \backslash P(G)_{wan})  = 0$. These notions stem from the
theory of dynamical systems where they describe important properties
of the measure $m$ with respect to the left shift on $P(G)$. The aim is here to show how these
properties can be determined from the matrix $A(\beta)$.

We will say that $A(\beta)$ is \emph{recurrent}
when $\sum_{n=0}^{\infty} A(\beta)^n_{vv} = \infty$ for some $v \in
V$. 
The inequality
\begin{equation}\label{comp13}
A(\beta)^l_{vu} A(\beta)^n_{uw}A(\beta)^k_{wv} \leq
A(\beta)^{n+k+l}_{vv},
\end{equation}
valid for all $n,k,l \in \mathbb N$ and all $u,v,w \in V$, shows that
because $G$ is strongly connected $A(\beta)$ is recurrent if and only if $\sum_{n=0}^{\infty}
A(\beta)^n_{uw} = \infty$ for all $u,w \in V$. We
say that $A(\beta)$ is \emph{transient} when $\sum_{n=0}^{\infty}
A(\beta)^n_{vv} <\infty$ for one $v \in V$, or equivalently that $\sum_{n=0}^{\infty}
A(\beta)^n_{vw} <\infty$ for all $v,w \in V$.

\begin{thm}\label{h8} Assume that $G$ is strongly connected, and let $m$
  be a non-zero harmonic $\beta$-KMS measure on $P(G)$. Then $m$ is conservative if and only if
  $A(\beta)$ is recurrent, and dissipative if and only if
  $A(\beta)$ is transient.
\end{thm}
\begin{proof} Set $\psi_v = m(Z'(v))$ and note that $\psi$ is
  $A(\beta)$-harmonic by Lemma \ref{b5} and consider the
  corresponding matrix $B(\beta)$ from (\ref{Bbeta}). By Lemma
  \ref{a39} $B(\beta)$ is a matrix over $E$ in this case. It is easy
  to show (by induction) that
\begin{equation}\label{induction}
B(\beta)^n_{e,f}= \psi_{r(e)}^{-1}\psi_{r(f)} e^{-\beta F(f)}
A(\beta)^{n-1}_{r(e),s(f)}
\end{equation}
for all $n \geq 1$. With this as one key of translation, much of the proof can be
completed by use of standard results from the theory of Markov chains as follows. We use Theorems 3.2 and 3.4 in \cite{Wo} to the
Markov chains on $E$ arising from $B(\beta)$. For this note that when we consider an edge
$e\in E$ as a state $x$ for such a Markov chain, the measure $\text{Pr}_x$ in
\cite{Wo} is the measure $\text{Pr}_e$ used in the proof in Lemma
\ref{boet}, cf.(\ref{cyl24}), and the quantity $H(x,y)$ in
\cite{Wo} is
$$
\text{Pr}_x\left(  \left\{ (x_i)_{i =1}^{\infty} \in P(G): \ x_i = y
   \ \text{for infinitely many} \ i  \right\}  \right)  .
$$
Equally important is it that the quantity $G(x,x)$ occurring in
Theorem 3.4 of \cite{Wo} in the present application is the sum $\sum_{n=0}^{\infty}
B(\beta)^n_{x,x}$. Once this is realized the rest of the 'translation'
reads a follows.

Assume that $m$ is conservative. Fix an edge $e \in E$. Since $m$ is conservative
it is supported on the set (\ref{rec7}) and since $m = m_{\psi}$ by
Theorem \ref{a17}, it
follows that so is the measure $\text{Pr}_e$, and the edge $e$ is therefore recurrent as a
 state for the Markov chain defined from $B(\beta)$, cf. Definition
 3.1 in \cite{Wo}. It follows first from Theorem 3.4 (a) in
 \cite{Wo} that $\sum_{n=0}^{\infty} B(\beta)^n_{e,e} = \infty$ and
 then from (\ref{induction}) (and Lemma \ref{a39}) that $A(\beta)$
 is recurrent. Conversely, if $A(\beta)$ is recurrent it follows from
 (\ref{induction}) and Theorem 3.4 (a) in \cite{Wo} that $e$ is a
 recurrent state for the Markov chain. Since $G$ is strongly connected
 it follows then from Theorem
 3.4 (b) in \cite{Wo} that 
$$
\text{Pr}_e\left(\left\{ (x_i)_{i =1}^{\infty} \in Z'(e): \ x_i = f
   \ \text{for infinitely many} \ i  \right\}\right) = 1
$$
for all $f\in E$. But $E$ is countable and $\text{Pr}_e$ is a probability
measure on $Z'(e)$ and it follows therefore that
\begin{equation}\label{combine1}
\text{Pr}_e\left( \bigcap_{f \in E} \left\{ (x_i)_{i =1}^{\infty} \in Z'(e): \ x_i = f
   \ \text{for infinitely many} \ i  \right\}\right) = 1 .
\end{equation}
Now recall from the proof of Lemma \ref{a20} that
\begin{equation}\label{combine2}
m(B) = m_{\psi}(B) = \sum_{v \in V} \sum_{e \in s^{-1}(v)} e^{-\beta
  F(e)}\psi_{r(e)} \text{Pr}_e\left(B \cap Z'(e)\right) 
\end{equation}
for all Borel sets $B$. By combining (\ref{combine1}) and
(\ref{combine2}) it follows that $m\left(P(G) \backslash
  P(G)_{rec}\right) = 0$, i.e. $m$ is conservative.     

Assume then that $m$ is dissipative. Note that 
$$
P(G)_{wan} \ \subseteq \ \bigcap_{e \in E } \left\{ (x_i)_{i=1}^{\infty} \in
  P(G) : \ \# \left\{ i \in \mathbb N: \ x_i = e\right\} < \infty  \right\} .
$$
By combining this with the relation
(\ref{combine2}) it follows from Theorem 3.2
(b) in \cite{Wo} that every edge $e\in E$ is a transient state for the Markov chain defined from
$B(\beta)$. By Theorem 3.4 (a) in \cite{Wo} this means that $\sum_{n
  =0}^{\infty} B(\beta)^n_{e,e} < \infty$ for all $e \in E$, and then
$A(\beta)$ is transient by (\ref{induction}). 

The converse implication, that transience of $A(\beta)$ implies
dissipativity of $m$, is stronger than what can be obtained directly from
\cite{Wo} by the method used so far and we give instead a direct
proof. Consider two vertexes $w,v$ and the set
$$
M(w,v) = \left\{ (x_i)_{i=1}^{\infty} \in Z'(w)
  : \  s(x_j) = v \ \text{for infinitely many} \ j \right\} .
$$ 
Let $N \in \mathbb N$ and for $j \geq N$, let
$$
M^j =  \left\{ (x_i)_{i=1}^{\infty} \in Z'(w)
  : \  s(x_j) = v \  \right\}.
$$
Then $M(w,v)  \subseteq \bigcup_{j \geq N} M^j$ and hence
$$
m(M(w,v)) \leq \sum_{j=N}^{\infty} m\left(M^j\right) .
$$
From Lemma \ref{b101} we get the identity $m(M^j) =
A(\beta)^{j-1}_{wv} m(Z'(v))$ and therefore the estimate
$$
m(M(w,v)) \ \leq \ m(Z'(v))\sum_{j=N}^{\infty} A(\beta)^{j-1}_{wv} .
$$
Since $\lim_{N \to \infty} \sum_{j=N}^{\infty} A(\beta)^{j-1}_{wv} = 0$
because $A(\beta)$ is transient, we
conclude that $m(M(w,v)) = 0$. This completes the proof because
$$
P(G) \backslash P(G)_{wan} \ = \ \bigcup_{w,v \in V} M(w,v).
$$

\end{proof}

Note that Theorem \ref{h8} describes a dichotomy: For any $\beta$
either all non-zero harmonic
$\beta$-KMS measures are dissipative or they are all conservative.

\subsection{Criteria for recurrence and transience}
In Proposition \ref{renew5} below we describe a recurrence/transience criterion which will be very
useful in the following. It is essentially equivalent to part (i)
of Theorem C in \cite{V}, but while Vere-Jones omits the proof by
referring to standard renewal arguments, we give a direct proof, parts
of which will be used later.   

Let $v\in V$. A \emph{loop} based at $v$ in $G$ is a finite path $\mu = e_1e_2e_3\cdots e_n$ of
length $n \geq 1$ such that $s(\mu) = r(\mu) = v$. It is a \emph{simple
  loop} when $r(e_i) \neq r(\mu), i = 1,2,\cdots , n-1$. The set of loops
(resp. simple loops)
$\mu$ of length $n$ based at $v$ will be denoted by $L_v(n)$
(resp. $l_v(n)$). We set
$$
l^n_{vv}(\beta) = \sum_{\mu \in l_v(n)} e^{- \beta F(\mu)} 
$$
when $l_v(n) \neq \emptyset$ and $l^n_{vv}(\beta) = 0$ when $l_v(n) = \emptyset$. 

\begin{lemma}\label{renew1} For $t> 0$,
$$
 \sum_{k=0}^{\infty} \left( \sum_{j=1}^{\infty}
  l^j_{vv}(\beta)t^j \right)^k = \sum_{n=0}^{\infty}
A(\beta)^n_{vv}t^n.
$$ 

\end{lemma} 
\begin{proof} We claim that
\begin{equation}\label{renew14}
\sum_{n=1}^N A(\beta)^n_{vv}t^n \ \leq \ \sum_{k=1}^N \left( \sum_{j=1}^N
  l^j_{vv}(\beta)t^j \right)^k  \ \leq \ \sum_{n=1}^{N^2} A(\beta)^n_{vv}t^n
\end{equation}
for all $N \in \mathbb N$. The proof is a matter of book-keeping.
For each $(d_1,d_2, \cdots, d_j) \in  \left\{1,2,\cdots ,N\right\}^j$,
set
$$
W(d_1,d_2, \cdots , d_j) = l^{d_1}_{vv}(\beta)t^{d_1} l^{d_2}_{vv}(\beta)t^{d_2}
\cdots l^{d_j}_{vv}(\beta)t^{d_j},
$$
and
$$
\left\{1,2,\cdots ,N\right\}^j_W = \left\{ \xi \in \left\{1,2,\cdots
    ,N\right\}^j: \ W(\xi) \neq 0 \right\}.
$$
Set
$$
A_N = \sqcup_{k=1}^N \left\{1,2,\cdots ,N\right\}_W^k ,
$$
and note that
\begin{equation}\label{renew4}
\sum_{k=1}^N \left( \sum_{j=1}^N
  l^j_{vv}(\beta)t^j \right)^k  = \sum_{\xi \in A_N} W(\xi) .
\end{equation}
An element $\mu \in \sqcup_{n=1}^NL_v(n)$ admits a unique decomposition $\mu = \mu_1\mu_2
  \cdots \mu_j$ such that each $\mu_i$ is a simple loop based at $v$, and in this
  way $\mu$ determines an element
$$
p(\mu) =  (\left|\mu_1\right|, \left|\mu_2\right|, \cdots ,
\left|\mu_j\right|) \in A_N .
$$
Since 
\begin{equation*}
p\left( L_v(n)\right) = \sqcup_{k=1}^N \left\{ (d_1,d_2,\cdots, d_k) \in
    \{1,2,\cdots, N\}_W^k : \ d_1 + \cdots +d_k = n \right\} 
\end{equation*}
and 
$$
\sum_{\mu \in p^{-1}(\xi)} e^{-\beta F(\mu)}t^{|\mu|} = W(\xi),
$$
we find that
$$
A(\beta)^n_{vv}t^n \ = \sum_{\mu \in L_v(n)} e^{-\beta F(\mu)}t^{|\mu|}  = \sum_{\xi}
W(\xi),
$$
where the last sum is over 
$$
\xi \in \sqcup_{k=1}^N \left\{ (d_1,d_2,\cdots, d_k) \in
    \{1,2,\cdots, N\}_W^k : \ d_1 + \cdots +d_k = n \right\} .
$$
Consequently
\begin{equation}\label{renew7}
\sum_{n=1}^N A(\beta)^n_{vv}t^n = \sum_{\xi \in B_N} W(\xi),
\end{equation}
where 
$$
B_N = \sqcup_{k=1}^N \left\{ (d_1,d_2,\cdots, d_k) \in
    \{1,2,\cdots, N\} _W^k : \ d_1 + \cdots +d_k \leq N\right\}.
$$ 
Then (\ref{renew14}) follows by combining (\ref{renew4})
and (\ref{renew7}) with the observation that $B_N \subseteq A_N
\subseteq B_{N^2}$. Define functions $f_N : \mathbb N \to [0,\infty)$ such
  that
$$
f_N(k) = \begin{cases}  \left(\sum_{j=1}^N
  l^j_{vv}(\beta) \right)^k,   & \ k \in \{1,2,\cdots , N\}, \\ 0, & k
> N . \end{cases}
$$
Since $f_N \leq f_{N+1}$, an application of Lebesgues monotone
convergence theorem to this sequence of functions shows that  
$$
\lim_{N \to \infty}  \sum_{k=1}^N \left( \sum_{j=1}^N
  l^j_{vv}(\beta) \right)^k  =  \sum_{k=1}^{\infty} \left( \sum_{j=1}^{\infty}
  l^j_{vv}(\beta) \right)^k .
$$
Combined with the inequalities from (\ref{renew14}) this leads to the
stated identity.
\end{proof}

\begin{prop}\label{renew5}  Assume that $G$ is strongly connected and
  let $\beta \in \mathbb R$. 
\begin{enumerate}
\item[$\bullet$] $A(\beta)$ is recurrent if and only if
  $\sum_{k=1}^{\infty} l^k_{vv}(\beta) \geq 1$ for some (and hence
  every) vertex $v\in V$.
\item[$\bullet$]  $A(\beta)$ is transient if and only if
  $\sum_{k=1}^{\infty} l^k_{vv}(\beta) < 1$ for some (and hence every)
  vertex $v \in V$. 
\end{enumerate} 
\end{prop}
\begin{proof} Take $t=1$ in Lemma \ref{renew1}.
\end{proof}

\subsection{The recurrent case} We still assume that
$G$ is strongly connected. In order to describe
for which $\beta$ there are $\beta$-KMS weights we introduce 
 the number (in $[-\infty,\infty]$),
\begin{equation}\label{h204}
\mathbb P(-\beta F) = \log \left( \limsup_n
  \left(A(\beta)^n_{vv}\right)^{\frac{1}{n}} \right) ,
\end{equation}
where $v$ is any vertex. Since $G$ is strongly connected it follows
from inequalities like (\ref{comp13}) that this quantity does not depend on
$v$ and is equal to
$$
\log \left( \limsup_n
  \left(A(\beta)^n_{uw}\right)^{\frac{1}{n}} \right)
$$
for any pair of vertexes $u,w$. The quantity $\mathbb P(-\beta F)$
appears in the theory of dynamical systems. In fact it is a version of what O. Sarig
calls the \emph{Gurevich pressure} in
\cite{S} and the function $\beta \mapsto \mathbb P(-\beta F)$ is
sometimes called the pressure function corresponding to a potential
defined from $F$. When $G$ is finite, and also in certain cases with
$G$ infinite, $\mathbb P(-\beta F)$ is the logarithm of the spectral radius of a
bounded operator defined from $A(\beta)$.

\begin{lemma}\label{h203} Assume that $G$ is strongly connected and
  that $C^*(G)$ is simple. Let $\beta \in \mathbb R$. There are no $\beta$-KMS weights unless $\mathbb P(-\beta F) \leq
  0$ and no boundary $\beta$-KMS weights unless $A(\beta)$ is transient.
 \end{lemma}
\begin{proof} If $A(\beta)$ is recurrent it follows that
  $\sum_{n=0}^{\infty} A(\beta)^n_{vw} = \infty$ for all $v,w \in V$
  and hence no infinite emitter can be $\beta$-summable. By Theorem
  \ref{a16} there are then no boundary $\beta$-KMS weights. Since
  $A(\beta)$ is recurrent if $\mathbb P(-\beta F) > 0$ this also
  implies that there are no boundary $\beta$-KMS weights unless $\mathbb P(-\beta F) \leq
  0$. If there is a harmonic $\beta$-KMS weight
  it follows from Theorem \ref{a17} that there is a non-zero
  $A(\beta)$-harmonic vector $\psi$. But then $A(\beta)^n_{vv}\psi_v \leq \sum_w A(\beta)^n_{vw}\psi_w \leq
\psi_v$ for all $n \in \mathbb N, \ v \in V$. Since $\psi_v > 0$ by
Lemma \ref{a39} it follows that 
$$
\limsup_n
  \left(A(\beta)^n_{vv}\right)^{\frac{1}{n}} = \limsup_n
  \left(A(\beta)^n_{vv}\psi_v\right)^{\frac{1}{n}} \leq \limsup_n
  \psi_v^{\frac{1}{n}}  = 1,
$$
implying that  $\mathbb P(-\beta F) \leq
  0$. It follows now from Theorems \ref{a1} and \ref{almost} that there are
no gauge invariant $\beta$-KMS weights when $\mathbb P(-\beta F) >
0$. By Proposition 5.6 in \cite{CT2} there are then no $\beta$-KMS weights
at all.
\end{proof}

\begin{thm}\label{recurrent} Assume that $G$ is strongly connected and
  that $C^*(G)$ is simple. Let $\beta \in \mathbb R$ and assume that
  $A(\beta)$ is recurrent. There exists a $\beta$-KMS weight for
  $\alpha^F$ if and only if $\mathbb P(-\beta F) =0$. When this is the
  case the $\beta$-KMS weight is unique up to multiplication by scalars, it
  is harmonic and the corresponding $\beta$-KMS measure on $P(G)$ is conservative.
\end{thm}
\begin{proof} Assume that there is a $\beta$-KMS weight for
  $\alpha^F$. Since $A(\beta)$ is transient when $\mathbb P(-\beta F)
  < 0$ it follows from Lemma \ref{h203} that $\mathbb P(-\beta F)
  =0$. Assume then that  $\mathbb P(-\beta F)
  =0$. By Lemma \ref{h203} there are no boundary $\beta$-KMS
  weights, but the work of Vere-Jones shows that there is a non-zero
  $A(\beta)$-harmonic vector, and that it is unique up to multiplication by
  scalars. This follow from Corollary 2 on page 371 in \cite{V}, and
  from Corollary \ref{summary} we conclude that there is a gauge
  invariant $\beta$-KMS weight for $\alpha^F$ which is unique up to
  multiplication by scalars. It follows
  from Theorem \ref{h8} that the corresponding $\beta$-KMS measure is
  conservative. Finally, by Proposition 5.6 in \cite{CT2} all $\beta$-KMS weights are gauge invariant.
\end{proof}

The following lemma gives a useful criterion for when the assumptions
in Theorem \ref{recurrent} are satisfied.

\begin{lemma}\label{renew9}  Assume that $G$ is strongly connected and
  let $\beta \in \mathbb R$. The following are equivalent:
\begin{enumerate}
\item[1)]  $\sum_{k=1}^{\infty} l^k_{vv}(\beta) = 1$.
\item[2)]  $A(\beta)$ is recurrent and $\mathbb P(-\beta F) =0$.
\end{enumerate}
\end{lemma} 
\begin{proof} 1) $\Rightarrow$ 2): The recurrence of $A(\beta)$ is a direct consequence of
  Proposition \ref{renew5} and 1). To show that $\mathbb P(-\beta F) =0$,
  note that $e^{-\mathbb P(-\beta F)}$ is the radius of convergence of the power
series $\sum_{n=0}^{\infty} A(\beta)^n_{vv} t^n$. By Lemma
\ref{renew1} this is the same as
$$
\sup \left\{ t \geq 0: \ \sum_{k=1}^{\infty} l^k_{vv}(\beta)t^k < 1
\right\} .
$$
It follows from 1) that this number is $1$.

2) $\Rightarrow$ 1): Since $A(\beta)$ is recurrent it follows from
Proposition \ref{renew5} that
$\sum_{k=1}^{\infty} l^k_{vv}(\beta) \geq 1$. Since we also assume
that $\mathbb P(-\beta F) =0$, which means that the radius of the
convergence of the power series $\sum_{n=0}^{\infty} A(\beta)^n_{vv}
t^n$ is $1$, it follows first that $\sum_{n=0}^{\infty} A(\beta)^n_{vv}
t^n < \infty$ and then by Lemma \ref{renew1} that $\sum_{k=1}^{\infty}
l^k_{vv}(\beta)t^k < 1$ when $t < 1$. From Lebesgues theorem on
monotone convergence we deduce from this that
$$
\sum_{k=1}^{\infty} l^k_{vv}(\beta) = \lim_{t \uparrow 1}
\sum_{k=1}^{\infty} l^k_{vv}(\beta)t^k \leq 1.
$$   
\end{proof}

In view of Theorem \ref{recurrent} we shall in the following focus on the case where
$A(\beta)$ is transient. Concerning the boundary KMS-weights the
situation is rather simple:

 \begin{prop}\label{trans0} Assume that $G$ is strongly connected and
  that $C^*(G)$ is simple. If $A(\beta)$ is transient there is a
  bijective correspondence between the infinite emitters in $G$ and
  the extremal rays of boundary $\beta$-KMS weights.
\end{prop}
\begin{proof}  Since $A(\beta)$ is
  transient it follows from 2) of Lemma \ref{a41} that all infinite
  emitters are $\beta$-summable and then from Theorem \ref{a16} that
  the infinite emitters correspond bijectively to the extremal rays of
  boundary $\beta$-KMS weights.
\end{proof}

\subsection{Finitely many non-wandering vertexes}

In this section we now consider the case where $G$ is strongly
connected and only has finitely
many vertexes. It
follows from 2) of Corollary \ref{b126} that there are no sinks in
$G$, but there may be infinite emitters in
$V$.

\begin{thm}\label{h202} Assume that $G$ strongly connected with
  finitely many vertexes and that $C^*(G)$ is simple. Let $\beta \in
  \mathbb R$. Assume that $A(\beta)$ is transient. The extremal rays
  of $\beta$-KMS weights for $\alpha^F$ are in bijective
  correspondence with the infinite emitters in $G$, and there are no harmonic $\beta$-KMS weights.  
\end{thm}

\begin{proof} As usual it follows from Proposition 5.6 in \cite{CT2}
  that all KMS-weights are gauge invariant, and in view of Proposition
  \ref{trans0} it suffices therefore to show that there are no non-zero
  $A(\beta)$-harmonic vectors. So assume $\psi$ is such a vector, and let
$v$ be a vertex in $G$. Then $\psi_v > 0$ by Lemma \ref{a39} and hence
$$
\sum_{w \in G} \sum_{n=0}^{\infty} A(\beta)^n_{vw}\psi_w =
\sum_{n=0}^{\infty}\sum_{w \in G} A(\beta)^n_{vw}\psi_w =
\sum_{n=0}^{\infty} \psi_v = \infty.
$$
It follows that $ \sum_{n=0}^{\infty} A(\beta)^n_{vw} = \infty$ for at
least one, and hence for all $w \in V$ since $G$ is strongly connected. In particular, $\sum_{n=0}^{\infty}
A(\beta)^n_{vv} = \infty$ which contradicts the assumed transience of $A(\beta)$.
\end{proof}

In the setting of Theorem \ref{h202} the algebra $C^*(G)$ is unital
and all
KMS-weights can be normalized to states.

Theorem \ref{h202} applies of course also when $G$ is a finite graph
which is strongly connected such that $C^*(G)$ is simple. In
combination with Theorem \ref{recurrent} the conclusion is that there
is at most one KMS state in this case. For the gauge action, where
$F=1$, this was first proved by Enemoto, Fujii and Watatani,
\cite{EFW}, and when $F$ is strictly positive (or strictly negative)
it follows from the work of Exel and Laca, \cite{EL}. For general $F$
it follows from \cite{CT1} where the necessary and sufficient
conditions for the existence of the KMS state is also described. As the next example shows the presence of
infinite emitters allows for a more complicated KMS spectrum.

\begin{example}\label{exx1} Let $G$ be the directed graph with two
  vertexes $v$ and $w$ and edges $e_i,f_i$ and $a,b$,
  such that $r(e_i) = s(e_i) = v$, $r(f_i) = s(f_i) = w$ for all $i\in
  \mathbb N$,
  while $s(a) = v, r(a) = w$ and $s(b) = w$ and $r(b) =v$. Define $F :
  E \to \mathbb R$ such that $F(e_{i}) = F(f_i) = \log (i+1)$ for
  all $i$ and $F(a) = F(b) = 1$. For $\beta > 1$, set
$$
s(\beta) =  \sum_{i=2}^{\infty}
  i^{-\beta} .
$$
Then 
$$
 l^i_{vv}(\beta) = \begin{cases} s(\beta), & \ i = 1, \\ e^{-2\beta} ,
   & \ i  = 2, \\ s(\beta)^{i-2}e^{-2\beta} , & \ i \geq 3,
\end{cases}
$$
and hence
$$
\sum_{k=1}^{\infty} l^k_{vv}(\beta) = s(\beta) + e^{-2 \beta} +
\frac{s(\beta)e^{-2\beta}}{1-s(\beta)} 
$$
when $s(\beta) < 1$. It follows that there is a real number $\beta_0 >
0 $
such that 
$\sum_{k=1}^{\infty} l^k_{vv}(\beta) < 1$ when $\beta >
\beta_0$, $\sum_{k=1}^{\infty} l^k_{vv}(\beta_0)  = 1$ and
$\sum_{k=1}^{\infty} l^k_{vv}(\beta)  > 1$ when $\beta
 < \beta_0$. Hence $A(\beta)$ is transient if and only if $\beta >
 \beta_0$ by Proposition \ref{renew5}. Theorem \ref{h202} now tells us that there are
two extremal $\beta$-KMS states for $\alpha^F$ when $\beta >
\beta_0$; one for each of two vertexes since they are both infinite
emitters. Theorem \ref{recurrent} implies that there are no
$\beta$-KMS states when $\beta < \beta_0$, and exactly one when $\beta
= \beta_0$.

\end{example}

We have now obtained a complete description and numbering of the $\beta$-KMS
weights 
when $G$ is strongly connected, $C^*(G)$ is simple and either
$A(\beta)$ is recurrent or $G$ only has finitely many vertexes. What
remains is therefore the cases where $A(\beta)$ is transient and $G$
has infinitely many vertexes. The only general result we have to offer
in that case is the following.

\begin{prop}\label{genbeta} Assume that $G$ is strongly connected with
  infinitely many vertexes and that
  $C^*(G)$ is simple. For $\beta \in \mathbb R$ there exist
  $\beta$-KMS weights for $\alpha^F$ if and only if $\mathbb P(-\beta
  F) \leq 0$.
\end{prop} 
\begin{proof} Thanks to Lemma \ref{h203} there is only one implication
  to consider. So assume that $\mathbb P(-\beta
  F) \leq 0$. If $A(\beta)$ is recurrent we have that $\mathbb P(-\beta
  F) = 0$ since $\sum_{n=0}^{\infty} A(\beta)^n_{vv} < \infty$ if $\mathbb P(-\beta
  F) < 0$. It follows therefore from Theorem \ref{recurrent} that
  there is a $\beta$-KMS weight in this case. Assume then that $A(\beta)$ is
  transient. If there are infinite emitters in $G$ the existence of a
  $\beta$-KMS weight follows from Proposition \ref{trans0}. Finally,
  when $A(\beta)$ is transient and $G$ does not have infinite emitters, it
  is a result of Pruitt, found as the Corollary at the bottom of page
  1799 in \cite{Pr}, that there is a non-zero $A(\beta)$-harmonic
  vector because $\mathbb P(-\beta F) \leq 0$. By Corollary
  \ref{summary} there is therefore a $\beta$-KMS weight also in this case. 
\end{proof}

We know now, at least in principle, when there are any $\beta$-KMS weights
to be found. In the following we shall identify a class of strongly
connected graphs where the structure of KMS weights is both transparent
and very rich.

\section{KMS weights from exits in $G$}

In this section we assume that $G$ is strongly connected and that
$A(\beta)$ is transient. Set
$$
P(V) = \left\{ (v_i)_{i=1}^{\infty} \in V^{\mathbb N} :  \ v_{i+1} \in
  r\left(s^{-1}(v_i)\right) \ \forall i \right\} .
$$
An element $t = (t_i)_{i=1}^{\infty} \in P(V)$ will be called an
\emph{exit path} when $\lim_{i \to \infty} t_i = \infty$, in the
sense that for every finite subset $M\subseteq V$ there is an $N \in
\mathbb N$ such that $t_i \notin M$ when $i \geq N$. Let $\beta \in \mathbb R$ and consider an exit path $t =(t_i)_{i=1}^{\infty}
\in P(V)$. Set
$$
t^{\beta}(i) =A(\beta)_{t_1t_2} A(\beta)_{t_2t_3}\cdots
A(\beta)_{t_{i-1}t_i} .
$$
Then $\sum_{n=0}^{\infty} A(\beta)^n_{v
  t_{i}} < \infty$ for all $i$ since $A(\beta)$ is transient and 
\begin{equation*}\label{n1}
 t^{\beta}(i)^{-1} \sum_{n=0}^{\infty} A(\beta)^n_{v
  t_{i}}  =t^{\beta}(i+1)^{-1} \sum_{n=0}^{\infty} A(\beta)^n_{v
  t_{i}}A(\beta)_{t_it_{i+1}}  \leq  t^{\beta}(i+1)^{-1} \sum_{n=0}^{\infty} A(\beta)^n_{v t_{i+1}}
\end{equation*}
for all $i$ and all $v \in V$. Hence the limit
\begin{equation}\label{n5}
\lim_{i \to \infty} t^{\beta}(i)^{-1} \sum_{n=0}^{\infty} A(\beta)^n_{v
  t_{i}} 
\end{equation}
exists, although it may of course be $+\infty$. For all $w \in V$ and
$k,i \in \mathbb N$ we have the estimate
$$
A(\beta)^k_{wv}  t^{\beta}(i)^{-1} \sum_{n=0}^{\infty} A(\beta)^n_{v
  t_{i}}  \leq  t^{\beta}(i)^{-1} \sum_{n=0}^{\infty}
A(\beta)^{n+k}_{w  t_{i}} \leq  t^{\beta}(i)^{-1} \sum_{n=0}^{\infty} A(\beta)^n_{w
  t_{i}}.
$$
Because $G$ is strongly connected it follows from this that the limit (\ref{n5})
is finite for all $v \in V$ if it is finite for one. When this holds
we say that $t$ is \emph{$\beta$-summable}.

\begin{lemma}\label{n6} Let $\beta \in \mathbb R$ and assume that
  $t$ is a $\beta$-summable exit path in
  $G$. It follows that the vector $\psi \in ]0,\infty[^{V}$ defined
  such that
\begin{equation*}\label{n9}
\psi_v =  \lim_{i \to \infty} t^{\beta}(i)^{-1} \sum_{n=0}^{\infty} A(\beta)^n_{v
  t_{i}} ,
\end{equation*}
is $A(\beta)$-harmonic.
\end{lemma}
\begin{proof} Note that $\psi_{t_1} \geq 1$ since
$$
 \sum_{n=0}^{\infty} A(\beta)^{n}_{t_1
  t_{i}} \geq A(\beta)^{i-1}_{t_1t_i} \geq t^{\beta}(i)
$$
when $i \geq 2$. Since $t^{\beta}(i)^{-1} \sum_{n=0}^{\infty} A(\beta)^n_{v
  t_{i}}$ never decreases when $i$ increases, we find that
\begin{equation*}
\begin{split}
&\sum_{w \in V} A(\beta)_{vw}\psi_w = \lim_{i \to \infty} \sum_{w \in V} A(\beta)_{vw}t^{\beta}(i)^{-1}\sum_{n=0}^{\infty} A(\beta)^n_{w
  t_{i}} \\
& 
= \lim_{i \to \infty}
t^{\beta}(i)^{-1} \sum_{n=0}^{\infty} \sum_{w \in V} A(\beta)_{vw} A(\beta)^n_{w
  t_{i}} = \lim_{i \to \infty}
 t^{\beta}(i)^{-1}\sum_{n=0}^{\infty} A(\beta)^{n+1}_{v
  t_{i}}.
\end{split}
\end{equation*} 
Since $t$ is an exit path, it follows that $A(\beta)^0_{vt_i} = 0$ and $\sum_{n=0}^{\infty} A(\beta)^{n+1}_{v
  t_{i}} = \sum_{n=0}^{\infty} A(\beta)^{n}_{v
  t_{i}}$ for all large $i$. Hence
$$
\sum_{w \in V} A(\beta)_{vw}\psi_w = \lim_{i \to \infty}
t^{\beta}(i)^{-1} \sum_{n=0}^{\infty} A(\beta)^{n+1}_{v
  t_{i}} = \psi_v.
$$

\end{proof}

It follows from Lemma \ref{n6} and Theorem \ref{a17} that a
$\beta$-summable exit path $t$ gives rise to a unique harmonic $\beta$-KMS
measure $m_t$ determined by the requirement that
$$
m_t(Z'(v)) =  \lim_{i \to \infty} t^{\beta}(i)^{-1} \sum_{n=0}^{\infty}
A(\beta)^n_{vt_i} 
$$
for all $v \in V$. We call this an \emph{exit measure}. To learn more
about this measure we shall use the shift map
$\sigma$ which acts on both $P(V)$ and $P(G)$ in the natural way; in both
cases given by the formula
$$ 
\sigma\left((x_i)_{i=1}^{\infty}\right) = (x_{i+1})_{i=1}^{\infty} .
$$
There is a map $\pi : P(G) \to P(V)$ defined by
$$
\pi\left( (e_i)_{i=1}^{\infty}\right) = \left(
  s(e_i)\right)_{i=1}^{\infty},
$$
which intertwines the two shift maps; $\sigma \circ \pi = \pi \circ
\sigma$.

\begin{lemma}\label{borexit} Let $t = (t_i)_{i=1}^{\infty} \in P(V)$. Then $\pi^{-1}(t)$ is a Borel subset of $P(G)$ and for
  any harmonic $\beta$-KMS measure $m$ on $P(G)$, any vertex $v\in V$ and
  any $n \in \mathbb N \cup \{0\}$,
$$
m \left( Z'(v) \cap \sigma^{-n}\left(\pi^{-1}(t)\right)\right) = \lim_{k \to \infty} A(\beta)^n_{vt_1}A(\beta)_{t_1t_2}
A(\beta)_{t_2t_3} \cdots A(\beta)_{t_kt_{k+1}}
m\left(Z'(t_{k+1})\right) .
$$
\end{lemma}
\begin{proof} For each $k$ let $M_k$ be the set of paths $\mu =
  e_1e_2\cdots e_k$ in $G$
  such that $s(e_i) = t_i, i = 1,2, \cdots,k$, and $r(e_k) =
  t_{k+1}$. Then
$$
\pi^{-1}(t) = \bigcap_{k \geq 1} \bigcup_{\mu \in M_k} Z'(\mu) ,
$$
proving that $\pi^{-1}(t)$ is Borel. Let $M_0$ be the set of paths $\nu$ of
length $n$ from $v$ to $t_1$. Then
$$
Z'(v) \cap \sigma^{-n}\left( \pi^{-1}(t)\right) = \bigcap_{k \geq 1} \left(
\bigcup_{\nu \in M_0,\mu \in M_k} \ Z'(\nu \mu)\right) .
$$
Furthermore, since 
$$
 \bigcup_{\mu \in M_{k+1}} Z'(\nu \mu) \subseteq  \bigcup_{\mu \in M_k}
   Z'(\nu\mu)
$$
for all $k$ and all $\nu \in M_0$, it follows that 
$$
m \left(Z'(v) \cap \sigma^{-n}\left( \pi^{-1}(t)\right)\right) =
\lim_{k\to \infty} m\left(   \bigcup_{\nu \in M_0, \mu \in M_k}
   Z'(\nu \mu)\right) .
$$
This completes the proof because it follows from (\ref{b102}) that
$$
 m\left(  \bigcup_{\nu \in M_0,\mu \in M_k} \ Z'(\nu \mu) \right) =  A(\beta)^n_{vt_1}A(\beta)_{t_1t_2}
A(\beta)_{t_2t_3} \cdots A(\beta)_{t_kt_{k+1}}
m\left(Z'(t_{k+1})\right) .
$$
\end{proof}

 Set 
$$
\mathcal G \pi^{-1}(t) \defeq  \bigcup_{n,m \in \mathbb N} \sigma^{-n}
\left( \pi^{-1} \left( \sigma^m(t)\right)\right) . 
$$

\begin{lemma}\label{h14} Let $t\in P(V)$ be an exit path, and let $m$ be a
  harmonic $\beta$-KMS measure on $P(G)$. Then
\begin{equation}\label{j1}
m(Z'(v) \cap \mathcal G \pi^{-1}(t)) = m(\pi^{-1}(t)) \lim_{i \to \infty} t^{\beta}(i)^{-1}\sum_{n=0}^{\infty} A(\beta)^n_{vt_i}
\end{equation}
for all $v \in V$.
\end{lemma}
\begin{proof} Note that $\sigma^{-n}\left(\sigma^i(\pi^{-1}(t))\right) \subseteq
\sigma^{-(n+1)} \left(\sigma^{i+1}(\pi^{-1}(t))\right)$ so that
$$
\bigcup_{n=0}^{\infty} \sigma^{-n}\left(\sigma^i(\pi^{-1}(t))\right) \subseteq
\bigcup_{n=0}^{\infty} \sigma^{-n}\left(\sigma^{i+1}(\pi^{-1}(t))\right) .
$$
Since 
$$
\mathcal G \pi^{-1}(t) = \bigcup_{i=1}^{\infty}
\bigcup_{n=0}^{\infty} \sigma^{-n}\left(\sigma^i(\pi^{-1}(t))\right) ,
$$
it follows that
\begin{equation}\label{j2}
m(Z'(v) \cap \mathcal G \pi^{-1}(t)) = \lim_{i \to \infty} m\left(Z'(v)
  \cap \bigcup_{n=0}^{\infty}
  \sigma^{-n}\left(\sigma^i(\pi^{-1}(t))\right)\right).
\end{equation}
Because $t$ is not pre-periodic under the shift, $\sigma^{-n}\left( \sigma^i(\pi^{-1}(t))\right) \cap
\sigma^{-n'}\left(\sigma^i(\pi^{-1}(t))\right) = \emptyset$ when $n
\neq n'$, and hence
\begin{equation}\label{j3}
 m\left(Z'(v)
  \cap \bigcup_{n=0}^{\infty}
  \sigma^{-n}\left(\sigma^i(\pi^{-1}(t))\right)\right) =
\sum_{n=0}^{\infty}   m\left(Z'(v)
  \cap \sigma^{-n}\left(\sigma^i(\pi^{-1}(t))\right)\right).
\end{equation}
Note that $\sigma^i\left( \pi^{-1}(t)\right) =
\pi^{-1}\left(\sigma^i(t)\right)$ and that $\sigma^i(t)$ is an exit
path because $t$ is. We can therefore use Lemma \ref{borexit} to conclude that
\begin{equation*}
 m\left(Z'(v) \cap \sigma^{-n}\left(\pi^{-1}(\sigma^k(t))\right)\right) =
  A(\beta)^n_{vt_{k+1}} t^{\beta}(k+1)^{-1} m(\pi^{-1}(t)) .
\end{equation*}
Inserted into (\ref{j3}) and (\ref{j2}) this yields (\ref{j1}).
\end{proof}

\begin{lemma}\label{n10} Let $t$ be a $\beta$-summable exit path
  in $P(V)$ and $m_t$ be the corresponding exit measure. Then 
$m_t(\pi^{-1}(t)) = 1$
and
$m_t$ is supported on $\mathcal G\pi^{-1}(t)$; that is,
$m_t\left(P(G) \backslash \mathcal G\pi^{-1}(t)\right) = 0$.
\end{lemma}
\begin{proof} Set
$$
U_n = \left\{ (x_i)_{i=1}^{\infty} \in P(G): \ s(x_i) = t_i, \ i \leq n
\right\} .
$$
Then $U_1 \supseteq U_2 \supseteq U_3 \supseteq \cdots$ is a
decreasing sequence of open sets in $P(G)$ such that
$$
\bigcap_j U_j = \pi^{-1}(t) .
$$  
It follows from Lemma \ref{b101} that 
$$
m_t(U_j) = t^{\beta}(j) \lim_{i \to \infty} t^{\beta}(i)^{-1} \sum_{n=0}^{\infty} A(\beta)^n_{t_j
  t_{i}}  .
$$
Since, for $i >j$,
$$
\sum_{n=0}^{\infty} A(\beta)^n_{t_j
  t_{i}} \geq A(\beta)_{t_jt_{j+1}}A(\beta)_{t_{j+1}t_{j+2}} \cdots
A(\beta)_{t_{i-1}t_i} = t^{\beta}(i){t^{\beta}(j)}^{-1}
$$
it follows that $m_t(U_j) \geq 1$ for all $j$. Combined with the
observation that $m_t(U_j) \leq
m_t\left(Z'(t_1)\right) < \infty$ for all $j$, we find that $1 \leq m_t\left(\pi^{-1}(t)\right) < \infty.$
Define a Borel measure $m$ on $P(G)$ such that
$$
m(B) =  m_t\left(\pi^{-1}(t)\right)^{-1} m_t\left(B \cap \mathcal
  G\pi^{-1}(t)\right) .
$$
Note that $m\left(\pi^{-1}(t)\right) = 1$ and that $m$ is supported on
$\mathcal G\pi^{-1}(t)$. Since $\mathcal G\pi^{-1}(t)$ is $\mathcal G$-invariant it follows from
Lemma \ref{h6} that $m$ is a harmonic $\beta$-KMS measure. It follows
therefore from Lemma \ref{h14} that
$$
m(Z'(v)) =  \lim_{i \to \infty} t^{\beta}(i)^{-1} \sum_{n=0}^{\infty} A(\beta)^n_{v
  t_{i}}  = m_t(Z'(v)) 
$$
for all $v\in V$. Hence $m = m_t$ by Theorem \ref{a17}.
\end{proof}

\begin{cor}\label{n101} Let $t$ be a $\beta$-summable exit path
  in $P(V)$, and let $m_t$ be the corresponding exit measure. Then
  $m_t$ is extremal among $\beta$-KMS measures. Specifically, when $\mu$ is a $\beta$-KMS
  measure such that $\mu \leq m_t$, then $\mu = sm_t$ where $s = \mu(\pi^{-1}(t))$.
\end{cor}
\begin{proof} Assume $\mu$ is a $\beta$-KMS measure and that $\mu \leq
  m_t$. Since $m_t$ is supported on $P(G)$ so is $\mu$, and $\mu$
  is therefore a harmonic $\beta$-KMS measure. Furthermore, $\mu$ is supported
  on $\mathcal G\pi^{-1}(t)$ since $m_t$ is and it follows then from
  Lemma \ref{h14} that $\mu(Z'(v)) = \mu(\pi^{-1}(t))
  m_t(Z'(v))$. This holds for all $v\in V$ and it follows from
  Theorem \ref{a17} that $\mu = \mu(\pi^{-1}(t))m_t$.
\end{proof}

Two elements
$t=(t_i)_{i=1}^{\infty}, t'=(t'_i)_{i=1}^{\infty}$ in $P(V)$ are \emph{tail
  equivalent} when there is a $k \in \mathbb Z$ such that $t_{i+k} =
t'_i$ for all large $i$, and \emph{tail inequivalent} otherwise. A tail equivalence class of exit paths will be
called an \emph{exit}. For $\beta \in \mathbb R$ we say that an exit
is $\beta$-summable when one of its exit paths is $\beta$-summable (and
then they all are).

\begin{prop}\label{h19} Assume that $G$ is strongly connected,
  that $C^*(G)$ is simple and that $A(\beta)$ is transient. Let $m$ be
  a harmonic $\beta$-KMS measure on $P(G)$. There is a decomposition $m = m_1 + m_2$ where $m_i,i = 1,2$, are harmonic
  $\beta$-KMS measures such that 
\begin{enumerate}
\item[i)]  there are a $N \in \mathbb N \cup \{0, \infty\}$, tail inequivalent $\beta$-summable exit paths
  $t^1,t^2,\cdots, t^N$ and positive real numbers $\lambda_1,\lambda_2, \cdots,
  \lambda_N$ such that 
$$
m_1 = \sum_{i=1}^N \lambda_i m_{t^i} ,
$$
and
\item[ii)]  $m_2 \circ \pi^{-1}(t) = 0$ for all $t  \in P(V)$.
\end{enumerate}
\end{prop}
\begin{proof} Set $M = \left\{t \in P(V): \ m\left(\pi^{-1}(t)\right)
    > 0 \right\}$. Since $m$ is dissipative by Theorem \ref{h8} all elements of $M$ are
  exit paths. Since $m(Z(v)) < \infty$ it follows that $M \cap
  \left\{t \in P(V): \ t_1 = v\right\}$ is countable for all $v$, and
  hence that $M$ is countable. Let $M = M_1 \sqcup M_2
  \sqcup \cdots $ be the partition of $M$ into tail
  equivalence classes and take an element $t^i =
  (t^i_j)_{j=1}^{\infty} \in M_i$ for each $i$. It follows from Lemma
  \ref{h14} that
$$
\lim_{j \to \infty} (t^i)^{\beta}(j)^{-1} \sum_{n=0}^{\infty}
A(\beta)^n_{t^i_1t^i_j} \ \leq \ \frac{m(Z'(t^i_1))}{m(\pi^{-1}(t^i))} ,
$$ 
showing that the $t^i$'s are all $\beta$-summable. Since $B \mapsto
m\left(B \cap \mathcal G\pi^{-1}(t^i)\right)$ is a harmonic $\beta$-KMS
measure by Lemma \ref{h6} it follows from a combination of Lemma
\ref{h14} and Lemma \ref{twice} that
$m\left(B \cap \mathcal G\pi^{-1}(t^i)\right) =
m\left(\pi^{-1}(t^i)\right)m_{t^i}(B)$ for all Borel sets $B$. Set
$\lambda_i = m\left(\pi^{-1}(t^i)\right)$ and note that 
$$
m_2(B) = m(B) -
  \sum_{i=1}^N \lambda_i m_{t^i}(B) = m\left( B \backslash
    \bigcup_{i=1}^N \mathcal G \pi^{-1}(t^i)\right) \ \geq \ 0
$$
for all Borel sets $B$. It follows from Lemma \ref{h6} that $m_2$ is a
harmonic $\beta$-KMS measure. This completes the proof because ii)
holds by construction.
\end{proof}

\begin{thm}\label{n11} Assume that $G$ is strongly connected, that
  $C^*(G)$ is simple and that there are at most countably many exits
  in $G$. Let $\beta \in \mathbb R$. There
  are no $\beta$-KMS weights unless $\mathbb P(- \beta F) \leq 0$. Assume
  therefore that this holds. 

\begin{enumerate}
\item[1)] Assume that $A(\beta)$ is recurrent. Then $\mathbb P(- \beta
  F) = 0$ and there is a
  $\beta$-KMS weight, unique up to multiplication by scalars. It is harmonic and the associated measure on
  $P(G)$ is conservative.
\item[2)] Assume that $A(\beta)$ is transient. The rays of
  extremal boundary $\beta$-KMS weights are in bijective
  correspondence with the infinite emitters in
  $G$, and all
  extremal harmonic $\beta$-KMS weights arise from exit measures of $\beta$-summable exits in $G$.
\end{enumerate}
\end{thm} 
\begin{proof}  That there
  are no $\beta$-KMS weights unless $\mathbb P(- \beta F) \leq 0$
  follows from Lemma \ref{h203}. Case 1) follows from Theorem
  \ref{recurrent}. The statement concerning boundary $\beta$-KMS
  weights in case 2) follow from Proposition \ref{trans0}. In view of
  Corollary \ref{n101} it suffices then to show that every non-zero extremal
  harmonic $\beta$-KMS measure $m$ is a multiple of the exit measure
  $m_t$ of a $\beta$-summable exit $t$. For this note that $P(G)_{wan}
  = \pi^{-1}(M)$ where $M$ is the set of exit paths in $P(V)$,
  cf. (\ref{wan7}). Since $G$ is a countable graph the set of exit
  paths tail-equivalent to a given exit path is countable and hence,
  under the present assumption, the set $M$ is countable. By Theorem
  \ref{h8}, $m$ is
  dissipative because $A(\beta)$ is transient, which means that $m(P(G) \backslash \pi^{-1}(M)) = 0$. Therefore, in the
  decomposition $m= m_1 + m_2$ from Proposition \ref{h19} the measure
  $m_2$ must be zero and by extremality of $m$ there is exactly one
  exit measure in the decomposition of $m_1 = m$.
\end{proof}

\section{Constructing graphs with prescribed  structure of KMS weights
  for the gauge action}

\subsection{The setting and the objective}

We retain the assumption that $G$ is strongly connected and consists
of more than a single loop so that $C^*(G)$ is simple. In addition we assume now that the function $F$ is
constant $1$ so that $\alpha^F$ is the gauge action on
$C^*(G)$. Let
$A$ be the adjacency matrix of $G$, cf. (\ref{h210}). Then $A(\beta) = e^{-\beta}A$ and 
\begin{equation}\label{n19}
\mathbb P(-\beta F) = \mathbb P(-\beta) = h(G) - \beta,
\end{equation}
where 
$$
h(G) = \limsup_n \frac{1}{n} \log A^n_{vv}  
$$
is independent of which vertex $v \in V$ we consider. The number $h(G)$ is
known as \emph{the Gurevich entropy} of $G$. It follows from
Lemma \ref{h203} that there are no KMS weights at all when $\beta <
h(G)$.  In the following we will therefore only consider
strongly connected graphs $G$ with finite Gurevich entropy. We will call these graphs \emph{admissible}. 

Following standard terminology, cf. e.g. \cite{Ru}, we will
say that $G$ is \emph{recurrent} when
$$
\sum_{n=0}^{\infty} A^n_{vv}e^{-nh(G)}
$$
is infinite for some and hence all $v\in V$, and that $G$ is
\emph{transient} when it is not. In the terminology we use here, $G$ is recurrent or
transient exactly when the matrix $e^{-h(G)}A$ is recurrent or transient, respectively.


 Let $t = (t_i)_{i=1}^{\infty}$ be an exit path in $G$ and consider
 $\beta \in \mathbb R$. Since we are
considering the gauge action we have that $A(\beta) = e^{-\beta}A$ and 
$$
t^{\beta}(k) = e^{-(k-1)\beta} t(k),
$$
where $t(k) = A_{t_1t_2}A_{t_2t_3}\cdots A_{t_{k-1}t_k}$, and
\begin{equation}\label{termk}
t^{\beta}(k)^{-1}\sum_{n=0}^{\infty} A(\beta)^n_{vt_k}  =  e^{(k-1)\beta}
t(k)^{-1} \sum_{n=0}^{\infty} A^n_{vt_k}e^{-n\beta}
\end{equation}
for all $\beta \in \mathbb R$ and all $v \in V$. For a given $\beta \in \mathbb R$
the exit path $t$ is $\beta$-summable if and only if
$$
\sum_{n=0}^{\infty} A^n_{t_1t_1} e^{-n\beta} < \infty
$$    
and
$$
\lim_{ k  \to \infty}  e^{(k-1)\beta}
t(k)^{-1} \sum_{n=0}^{\infty} A^n_{t_1t_k}e^{-n\beta}  < \infty .
$$
The convexity of the exponential function implies therefore that the
$\beta$-values for which an exit path is $\beta$-summable constitutes an interval. When there are $L$ exits, the possible inverse
temperatures which the exit measures and their KMS weights can
contribute is a set of the form
$$
\bigcup_{i=1}^L I_i 
$$
where each $I_i$ is a sub-interval (open, closed or half-open) in
$]h(G),\infty[$ when $G$ is recurrent and in $[h(G),\infty[$ when $G$
is transient. Furthermore, for each $\beta \in \bigcup_{i=1}^L I_i$
the number of extremal rays of $\beta$-KMS weights arising from the exits is
the number
$$
\# \left\{1 \leq i \leq L : \  \beta \in  I_i\right\}  .
$$

We aim now to show that all these possibilities can actually be
realized with one necessary restriction in the row-finite case.

\subsection{Ruette graphs}

When $v$ is a vertex we let $l^k_{vv}(G)$ be the number of simple
loops of length $k$ in $G$ based at $v$. We introduce now a class of
graphs which were considered by Ruette in Example 2.9 in \cite{Ru}.

\begin{defn}\label{Ruette} A vertex $v$ in an admissible graph $G$ is a
  \emph{Ruette vertex} when
\begin{enumerate}
\item[i)] $l^1_{vv}(G) = 1$, and
\item[ii)] $\limsup_n \left(l^n_{vv}(G)\right)^{\frac{1}{n}} =
  e^{h(G)}$.
\end{enumerate}
\end{defn}

An admissible graph with a Ruette vertex will be called a
\emph{Ruette graph}.

\begin{lemma}\label{ruette3} Let $G$ be an admissible graph and $v$ a vertex in
  $G$. Set 
$$
L = \limsup_n \left(l^n_{vv}(G)\right)^{\frac{1}{n}}. 
$$
The following are equivalent:
\begin{enumerate}
\item[i)] $L > 0$ and $\sum_{n=1}^{\infty} l^n_{vv}(G)L^{-n} = 1$.
\item[ii)]  $G$ is recurrent and $L =
  e^{h(G)}$. 
\end{enumerate}
\end{lemma}
\begin{proof} Apply Lemma
  \ref{renew9} to $G$ with $F= 1$ and $\beta = \log L$.
\end{proof}

\begin{prop}\label{ruette4} Let $G$ be an Ruette graph with Ruette
  vertex $v$. Assume that $G$ is recurrent and let $G'$ be the graph obtained
  from $G$ by removing the edge $e$ with $s(e) = r(e) = v$. Then $G'$
  is transient and $h(G') = h(G)$.
\end{prop}
\begin{proof} Since $G' \subseteq G$ we have trivially that $h(G')
  \leq h(G)$. Let $A'$ be the adjacency matrix of $G'$. Then
  $l^n_{vv}(G') \leq \left(A'\right)^n_{vv}$ and hence
$$
\limsup_n \left( l^n_{vv}(G')\right)^{\frac{1}{n}} \leq e^{h(G')} .
$$
But $l^n_{vv}(G') = l^n_{vv}(G)$ when $n \geq 2$ and $v$ is a Ruette
vertex in $G$ so
$$
\limsup_n \left( l^n_{vv}(G')\right)^{\frac{1}{n}} = e^{h(G)} .
$$
Therefore $h(G') = h(G)$. It follows from Lemma \ref{ruette3} that $\sum_{n=1}^{\infty}
l^n_{vv}(G) e^{-n h(G)} = 1$ and hence
$$
\sum_{n=1}^{\infty}
l^n_{vv}(G') e^{-n h(G')} = \sum_{n=2}^{\infty}
l^n_{vv}(G) e^{-n h(G)} = 1- e^{-h(G)} < 1.
$$
By Lemma \ref{renew1} this implies that $\sum_{n=0}^{\infty}
\left(A'\right)^n_{vv}e^{-n h(G')} < \infty$, i.e $G'$ is transient.

\end{proof}

\subsection{Technical lemmas on series}

We need some notation and terminology to describe the intervals of
summability for the exits in the examples we are going to
construct. Given three sequences $\textbf{a} = \{a_i\}_{i=1}^{\infty}$, $\textbf{b}
=\{b_i\}_{i=1}^{\infty}$ and $\textbf{c} =\{c_i\}_{i=1}^{\infty}$, of
non-zero natural numbers, let 
$$
J(\textbf{a},\textbf{b},\textbf{c})
$$ 
be the set
of real numbers $\beta \in \mathbb R$ with the property that
\begin{equation}\label{m6}
\sum_{k=1}^{\infty} \frac{b_k}{a_1a_2 \cdots a_k} e^{-k\beta} < \infty .
\end{equation}
and
\begin{equation}\label{m3}
\sum_{k=1}^{\infty} \frac{c_k}{a_1a_2\cdots a_{k}}
e^{k\beta} < \infty .
\end{equation}

\begin{lemma}\label{summint}  Let $ 0 < r \leq R$ be real numbers, and let $I$
  be one of the intervals 
$$
\big[r,R\big], \ \big]r,R\big], \ \big[r,R\big[, \ \big]r,R\big[,\
\big[r,\infty\big[, \ \big]r,\infty\big[. 
$$
There is a choice of sequences
$\textbf{a} = \{a_i\}_{i=1}^{\infty}$, $\textbf{b}
=\{b_i\}_{i=1}^{\infty}$ and $\textbf{c} =\{c_i\}_{i=1}^{\infty}$ of
non-zero natural numbers such
that $I = J(\textbf{a},\textbf{b},\textbf{c})$.
\end{lemma}

For the proof of Lemma \ref{summint} we need the following observation.

\begin{lemma}\label{m7} Let $\{q_n\}_{n=1}^{\infty},\{q'_n\}_{n=1}^{\infty}$ be sequences of
  positive rational numbers. There are sequences $\{a_n\}_{n=1}^{\infty}$, $\{b_n\}_{n=1}^{\infty}$,
  $\{c_n\}_{n=1}^{\infty}$ of non-zero natural numbers such that
$$
{q_n} = \frac{b_n}{a_1a_2\cdots a_n} \ \ \text{and} \ \ {q'_n} = \frac{c_n}{a_1a_2\cdots a_n}
$$
for all $n$.
\end{lemma}
\begin{proof} Left to the reader.
\end{proof}

\emph{Proof of Lemma \ref{summint}:}  We give only the details in the proof in the cases $ I =]r,R[$ and $I =
  ]r,R]$. The remaining four cases can be handled in a similar way. Set $S = e^R$ and $s = e^{-r}$. Let $\{\epsilon_n\}$
  and $\{\epsilon'_n\}$
  be sequences of positive real numbers and $\{q_n\}, \{q'_n\}$
  sequences of positive rational numbers such that
$$
\sum_{n=1}^{\infty} \epsilon_n s^n < \infty ,  \ \ \ 
\sum_{n=1}^{\infty} \epsilon'_n S^n < \infty,
$$
and
$$
\frac{1}{(s + n^{-1})^n} \leq \frac{1}{{s}^n} - \epsilon_n \leq q_n
\leq \frac{1}{s^n}, \ \ \frac{1}{(S + n^{-1})^n} \leq \frac{1}{{S}^n} -
\epsilon'_n \leq q'_n \leq \frac{1}{S^n}
$$
for all $n$. The radii
of convergence of the powers series $\sum_{n=1}^{\infty} q_n z^n$ and
$\sum_{n=1}^{\infty} {q'_n} z^n$ are $s$ and $S$, respectively. 
Note that 
$$
\sum_{n=1}^{\infty} q_n{s}^n \geq \sum_{n=1}^{\infty} \left(\frac{1}{{s}^n} -
    \epsilon_n\right){s}^n =  \infty .
$$
Similar, $\sum_{n=1}^{\infty} q'_n{S}^n = \infty$. It follows that if
we let $\textbf{a} = \{a_i\}_{i=1}^{\infty}$, $\textbf{b}
=\{b_i\}_{i=1}^{\infty}$ and $\textbf{c} =\{c_i\}_{i=1}^{\infty}$ be the
sequences obtained by applying Lemma \ref{m7} to the sequence
$\{q_n\}$ and $\{q_n'\}$, we have that $]r,R[ =
J(\textbf{a},\textbf{b},\textbf{c})$.

To handle the case $I = ]r,R]$ we proceed as follows: We use Lemma
\ref{m7} to find sequences $\textbf{a} = \{a_i\}_{i=1}^{\infty}$, $\textbf{b}
=\{b_i\}_{i=1}^{\infty}$ and $\textbf{c} =\{c_i\}_{i=1}^{\infty}$ such that
$$
q_n = \frac{b_n}{a_1a_2\cdots a_n},
$$
and
$$
\frac{q'_n}{n^2} = \frac{c_n}{a_1a_2\cdots a_n}.
$$
The series $\sum_{n=1}^{\infty} \frac{q'_n}{n^2} z^n$ has $S$ as
its radius of convergence, and
$$
\sum_{n=1}^{\infty} \frac{q'_n}{n^2}{S}^n \leq \sum_{n=1}^{\infty}
\frac{1}{n^2} < \infty.
$$  
It follows that (\ref{m6}) holds if and only $e^{-\beta} < s$ while
(\ref{m3}) holds if and only if $e^{\beta} \leq S$. Hence
$J(\textbf{a},\textbf{b},\textbf{c}) = ]r,R]$.  
\qed

\begin{lemma}\label{djurs10} Let $D \in ]0,1[$ and let $\{a_n\}_{n=2}^{\infty}$ be a
  sequence of non-negative integers such that
\begin{equation}\label{less1}
\sum_{n=2}^{\infty} a_n D^n < 1-D.
\end{equation}
There is a sequence $\{b_n\}_{n=1}^{\infty}$ of non-negative integers such
that 
\begin{enumerate}
\item[i)] $b_1 = 1$,
\item[ii)] $b_n \geq a_n$ for all $n \geq 2$, 
\item[iii)] $b_n - a_n \geq 2$ for infinitely many $n$,
\item[iv)] $\limsup_n (b_n)^{\frac{1}{n}} =
D^{-1}$ and 
\item[v)] $\sum_{n=1}^{\infty} b_n D^n = 1$.
\end{enumerate}
\end{lemma} 
\begin{proof} Set $s = 1- D -\sum_{n=2}^{\infty}a_n D^n$ and choose a
  sequence of real numbers $0 < s_1 <s_2 <s_3 < \cdots$ such that $\lim_{m \to\infty}
  s_m = s$. Choose first a sequence $\{r_n\}_{n=1}^{\infty}$ in $\mathbb N$ such that 
\begin{equation}\label{rchoice}
\lim_{m \to  \infty} (s_m -s_{m-1})^{\frac{1}{r_m}} = 1.
\end{equation}
We will inductively construct sequences of natural numbers
  $\{k_m\}_{m=1}^{\infty}$ and $\{n_m\}_{m=1}^{\infty}$ such that
  $\{n_m\}$ is strictly increasing and
\begin{enumerate}
\item[a)] $n_{m} \geq r_m $ for all
  $m\geq 1$,
\item[b)] $ k_{m} D^{n_m} \geq
  s_{m} - s_{m-1}$ for all $m \geq 2$, and 
\item[c)] $0 < s_m  - \sum_{i=1}^{m}  k_i D^{n_i}  \leq
\frac{s}{m}$ for all $m\geq 1$.
\end{enumerate} 
To start the induction note that since $\lim_{n \to \infty} D^{n} =
 0$, we can choose $n_1,k_1\in \mathbb N$ such that $n_1 \geq \max\{r_1,2\}$ and
$k_1 D^{n_1} < s_1$. Assume then that $n_i,k_i$ have been found when $i \leq m-1$. To simplify notation, set
$$
\Sigma_{m-1} = \sum_{i=1}^{m-1}  k_i D^{n_i}  .
$$
Note that
\begin{equation}\label{djurs14}
\begin{split}
& s_{m} -s_{m-1} + \frac{m-1}{m}\left(s_{m-1} - \Sigma_{m-1}\right) \\
& =s_{m} - \Sigma_{m-1} - \frac{1}{m}\left(s_{m-1}
  -\Sigma_{m-1}\right)  < s_{m} - \Sigma_{m-1} .
\end{split}
\end{equation}

Using that  $\lim_{n \to \infty} D^{n} = 0$
we find an $n_{m} > \max\{ r_{m}, n_{m-1}\}$ in $\mathbb N$ such that
\begin{equation}\label{djurs11}
3D^{n_{m}} \leq  \frac{1}{m}\left(s_{m-1}
  -\Sigma_{m-1}\right) . 
\end{equation}
Let $k_{m}$ be the least natural number such that
\begin{equation}\label{djurs13}
 k_{m} D^{n_{m}} > s_{m}-s_{m-1} + \frac{m-1}{m}(s_{m-1} -\Sigma_{m-1}) .  
\end{equation}
Then b) certainly holds, and since $k_{m}$ is
the least natural number with the stated property it follows by use of
(\ref{djurs11}) and (\ref{djurs14}) that
$$
  k_{m} D^{n_{m}} <  s_{m} - \Sigma_{m-1},
$$
giving us the strict inequality in c). The other inequality in c)
follows because
$$
s_{m} - \Sigma_{m-1} - \frac{1}{m}\left(s_{m-1} -\Sigma_{m-1}\right)  <
k_{m} D^{n_{m}} <  s_{m} - \Sigma_{m-1} 
$$
by construction.

Set $c_n = k_m$ when $n = n_m$ and $c_n = 0$ when $ n \notin \left\{n_k: \ k \in \mathbb
  N \right\}$, and then $b_n = c_n + a_n, n \geq 2$, and $b_1 = 1$
. Then i), ii) and iii) hold and v) follows from c):
$$
\sum_{n=1}^{\infty} b_n D^n = D + \sum_{n=2}^{\infty} a_nD^n +
\sum_{m=1}^{\infty} k_mD^{n_m} = D+
\sum_{n=2}^{\infty} a_nD^n + s = 1.
$$  
To see that iv) holds note first that v) implies that $\limsup_n
(b_n)^{\frac{1}{n}} \leq 
D^{-1}$. To establish the reversed inequality note that b) implies that 
$$
\limsup_m
(k_m)^{\frac{1}{n_m}} D = \limsup_m
\left(k_mD^{n_m}\right)^{\frac{1}{n_m}} \geq \limsup_m (s_m
-s_{m-1})^{\frac{1}{n_m}}.
$$ 
Thanks to (\ref{rchoice}) and because $n_m
\geq r_m$, it follows that $ \limsup_m (s_m
-s_{m-1})^{\frac{1}{n_m}} =1$ and hence that $\limsup_m
(k_m)^{\frac{1}{n_m}} \geq D^{-1}$. This finishes the proof because
$$
\limsup_n (b_n)^{\frac{1}{n}} \geq \limsup_n (c_n)^{\frac{1}{n}} =
\limsup_m (k_m)^{\frac{1}{n_m}} .
$$
 \end{proof}

\section{The row-finite cases}

In this section we prove the following two theorems.

\begin{thm}\label{REV1} Let $h
  \in ]0,\infty[$ be a positive real number. Let $\mathbb I$ be a finite or
  countably infinite collection of intervals in $\big]h,\infty\big[$. Assume that
$$
I_0 = \big]h,\infty\big[
$$ 
for at least one $I_0 \in \mathbb I$.

There is a strongly connected
recurrent row-finite graph $G$ with Gurevich entropy $h(G) =h$ such
that the set of exits in $G$ is in bijective correspondence with
$\mathbb I$. Furthermore, for $\beta \geq
  h$ there are the following extremal $\beta$-KMS weights for the
  gauge action on $C^*(G)$:
\begin{enumerate}
\item[$\bullet$] For $\beta > h$ the rays of extremal $\beta$-KMS
  weights are in bijective correspondence with the set
$$
\left\{ I \in \mathbb I :  \ \beta \in I \right\} .
$$
\item[$\bullet$] For $\beta = h$ there is a unique ray of extremal $h$-KMS weights.
\end{enumerate}
\end{thm}

\begin{thm}\label{REV2} Let $h
  \in ]0,\infty[$ be a positive real number. Let $\mathbb I$ be a finite or
  countably infinite collection of intervals in $\big[h,\infty\big[$. Assume that
$$
I_0 = \big[h,\infty\big[
$$ 
for at least one $I_0 \in \mathbb I$.

There is a strongly connected
transient row-finite graph $G$ with Gurevich entropy $h(G) =h$ such
that the set of exits in $G$ is in bijective correspondence with $\mathbb I$. Furthermore, for $\beta \geq
  h$ the rays of extremal $\beta$-KMS weights for the gauge action on
  $C^*(G)$ are in bijective correspondence with the set
$$
\left\{ I \in \mathbb I :  \ \beta \in I \right\} .
$$
\end{thm}

In order to fully appreciate these results it must be observed that the
assumed presence of an interval in $\mathbb I$ of maximal size is not
accidental. To explain this we introduce a key ingredient in the
following constructions. An exit path $(v_i)_{i=1}^{\infty}$ in $G$ is a \emph{bare exit path} when
$s\left(r^{-1}(v_{i+1})\right) = v_i$ for all $i$. It is easy to see
that a bare exit is $\beta$-summable for all $\beta$ in the largest
possible interval, cf. Lemma \ref{r11}, namely for all $\beta$ such that $e^{-\beta}A$ is
transient, where
$A$ denotes the adjacency matrix of the graph. And it is a fact, proved in Appendix \ref{appendix}, that a strongly
connected graph with infinitely many vertexes and at most countably many exits must have a bare
exit when it is row-finite. So the presence of a largest possible
interval among the intervals in $\mathbb I$ can not be avoided; for
strongly connected row-finite graphs with infinitely many vertexes and
at most countably many exits
there must be an exit which contributes a $\beta$-KMS weight for
the gauge action for all $\beta$ in the largest possible interval.

\begin{lemma}\label{r11}  Assume that $G$ is strongly connected with
  adjacency matrix $A$. A bare exit path $t = (t_i)_{i=1}^{\infty}$
  is $\beta$-summable if and only if $\sum_{n=0}^{\infty} A^n_{t_1t_1}
  e^{-n \beta} < \infty$. 
\end{lemma} 
\begin{proof} Assuming that $\sum_{n=0}^{\infty} A^n_{t_1t_1}
  e^{-n \beta} < \infty$ we observe that
$$
\sum_{n=0}^{\infty} A^n_{t_1t_k} e^{-n\beta} =  t(k)e^{-(k-1)\beta}
\sum_{n = 0}^{\infty} A^{n}_{t_1t_1} e^{-n\beta}
$$
since $t$ is bare. Hence
$$
e^{(k-1)\beta} t(k)^{-1}\sum_{n=0}^{\infty} A^n_{t_1t_k} e^{-n\beta} =
\sum_{n = 0}^{\infty} A^{n}_{t_1t_1} e^{-n\beta} < \infty
$$
for all $k$, and $t$ is $\beta$-summable. The converse is trivial. 

\end{proof}

\subsection{The basic construction I:  Adding an exit}\label{constr} In the following 
construction of examples a key step is the attachment of an exit
path to an existing graph by a procedure which depends on whether 
the resulting graph should be row-finite or not. In this section we
describe the relevant method in the row-finite case.

Let $G_0$ be an arbitrary directed graph. We assume that
$(v_i)_{i=1}^{\infty}$ is a bare exit in $G_0$. Let $I$ be an arbitrary
interval in $]0,\infty[$ and choose with the aid of Lemma \ref{m7}
three sequences $\textbf{a} = (a_i)_{i=1}^{\infty},\textbf{b} = (b_i)_{i=1}^{\infty}
  , \textbf{c} = (c_i)_{i=1}^{\infty}$ of non-zero natural numbers
  such that $J(\textbf{a},\textbf{b},\textbf{c}) = I$. For each $i
  \geq 1$ we let $d_i = c_ia_{i+1}a_{i+2}\cdots a_{2i}$. Add to $G_0$
  the vertexes $t_2, t_3,t_4, \cdots $ and set $t_1 = v_M$, where $M
  \in \mathbb N$ is arbitrary. Add $a_i$ arrows from $t_i$ to $t_{i+1}$, add
  $d_i$ arrows from $v_{M+i-1}$ to $t_{2i+1}$ and $b_i$ arrows from
  $v_{M-1+2i}$ to $t_{i+1}$ for $i =1,2,3,\cdots $. Let $G$ denote the
  resulting graph. If we ignore the part of $G_0 \backslash \{v_i\}_{i=1}^{\infty}$ with no contact to
  the new vertexes, the graph $G$ looks at follows.
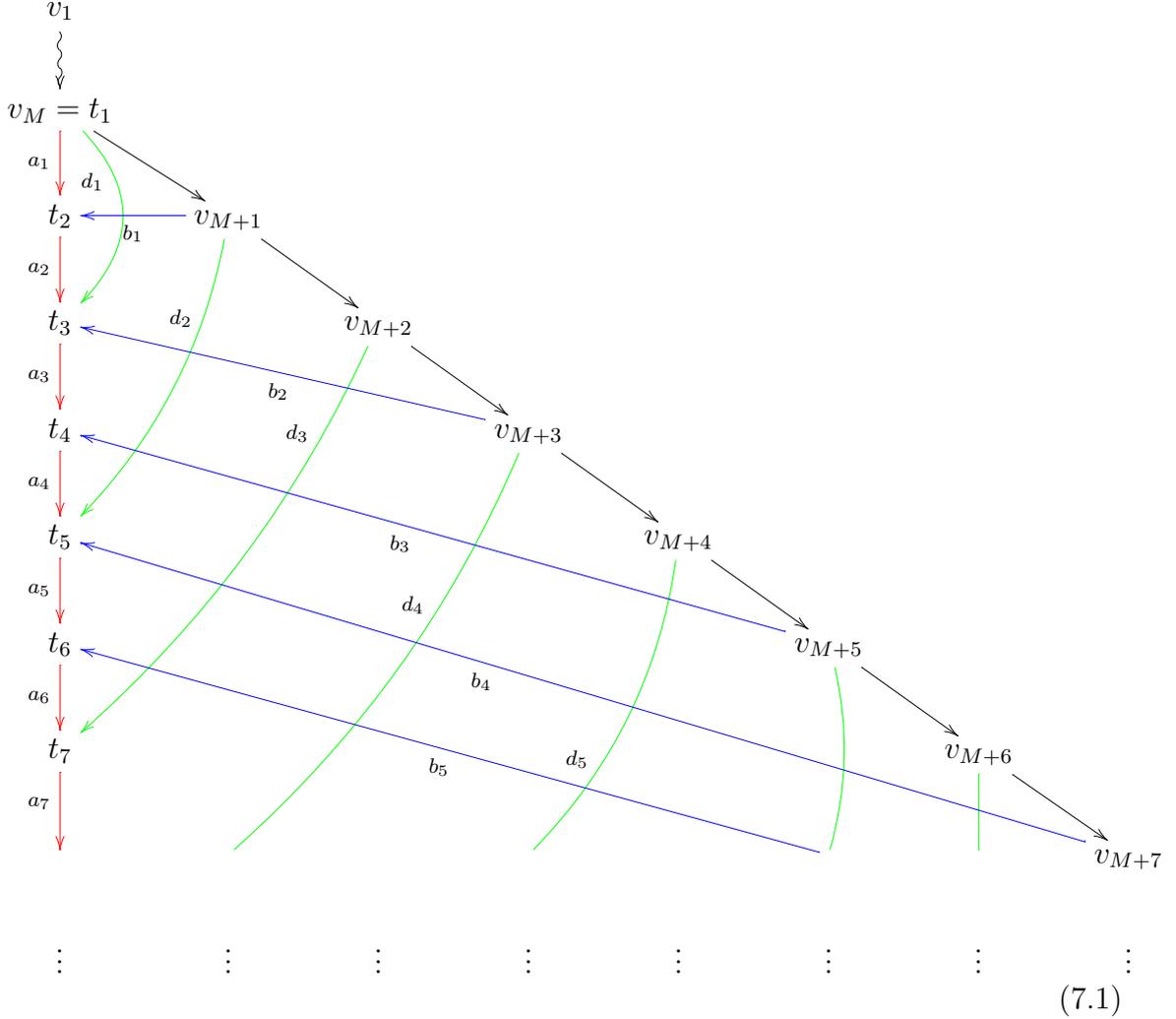
\begin{equation}\label{G1}
\begin{xymatrix}{ v_1  \ar@{~>}[d] 
  & & & & && &&&\\ 
 v_M =t_1 \ar@[red][d]_{a_1}  \ar[rd] \ar@[green]@/^2pc/[dd]_(0.30){d_1}  & & & & & & & & &\\
 t_2  \ar@[red][d]_{a_2}&
 \ar@[blue][l]^-{b_1} \ar[rd] v_{M+1}  \ar@[green]@/^1pc/[dddl]_(0.3){d_2} &  & & & & & && \\
    t_3\ar@[red][d]_{a_3}
&  &  \ar[rd] \ar@[green]@/^1pc/[ddddll]_(0.25){d_3} v_{M+2}
 &  & & & & && \\
    t_4 \ar@[red][d]_{a_4}
&  & & \ar@[blue][ulll]^-{b_2} \ar[rd]
\ar@{-}@[green]@/^1pc/[ddddll]_(0.4){d_4} v_{M+3} & & & & && \\
   t_5 \ar@[red][d]_{a_5}
&  &  & & \ar[rd] v_{M+4} \ar@{-}@[green]@/^1pc/[dddl]_(0.7){d_5}   & & & &&\\
 t_6 \ar@[red][d]_{a_6}  &
 &  &  & &  \ar@[blue][llllluu]^(0.55){b_3}  v_{M+5} \ar@{-}@[green]@/^/[dd] \ar[dr] & & && \\
  t_7 \ar@[red][d]_{a_7} &
 &  &  & & &\ar[dr] v_{M+6} \ar@{-}@[green][d] &  && \\
  &  &  &  & &   \ar@[blue][llllluu]^(0.5){b_5}  & &
  \ar@[blue][llllllluuu]^(0.6){b_4} v_{M+7} && \\
 \vdots  & \vdots &  \vdots &  \vdots &
 \vdots & \vdots & \vdots & \vdots &  &   \\ }
\end{xymatrix}
\end{equation}

The vertex $v_M$ will be referred to as \emph{the first contact
  vertex}.

\bigskip

\subsection{Proof of Theorems \ref{REV1} and \ref{REV2}:} 
Let $I_1,I_2,I_3, \cdots$ be a numbering of the
  intervals in $\mathbb I \backslash \{I_0\}$. Start with the following graph $H_0$:
\begin{equation}\label{feber}
\begin{xymatrix}{ v_1 \ar[r] & v_2 \ar[r]  & v_3 \ar[r]    & v_4
    \ar[r] & v_5 \ar[r] & \hdots   \hdots}
\end{xymatrix}
\end{equation}
the vertexes of which is a bare exit path $t^0$. Using this exit path as a backbone we shall build a
recurrent Ruette graph $G$ with $v_1$ as a Ruette vertex such that $G$
has the properties described in Theorem \ref{REV1} while the graph
$G'$ obtained from $G$ by removing the edge starting and ending at
$v_1$ will be a graph with the properties described in Theorem
\ref{REV2}.

Attach to $H_0$ an exit path $t^1$ as described in Section \ref{constr} using $v_1$ as
first contact vertex and using multiplicities $\textbf{a(1)} =
\{a_i\},\textbf{b(1)} = \{b_i\}$ and $\textbf{c(1)} = \{c_i\}$ such that
$J(\textbf{a(1)},\textbf{b(1)},\textbf{c(1)}) = I_1$. Let $H_1$ be the
resulting graph. Attach to $H_1$ an exit path $t^2$ using $v_2$ as
first contact vertex and using multiplicities $\textbf{a(2)},\textbf{b(2)}$ and $\textbf{c(2)}$ such that
$J(\textbf{a(2)},\textbf{b(2)},\textbf{c(2)}) = I_2$. Let $H_2$ be the
resulting graph and continue recursively, either indefinitely or until
the intervals in $\mathbb I\backslash \{I_0\}$ have been used up. Set
$$
H = \bigcup_i H_i .
$$
Since we used different first contact vertexes the graph $H$ will be
row-finite. In case $I_0$ is the only interval in $\mathbb I$ we
haven't done anything and the
graph $H$ is just the bare exit $H_0$ we started with. In any case, $H$ is
certainly not strongly connected.

We next add to $H$ a sequence $\mu_i , i =1,2,3, \cdots $, paths going
from
vertexes in the exit paths we have attached to the vertex $v_1$, in order to make the
resulting graph $G_0$ strongly connected. The paths are mutually disjoint, both
from each other and from $H$, except for their initial and terminal
vertexes. We want also to arrange that
\begin{equation}\label{feber1}
\sum_{n=1}^{\infty} l^n_{v_1v_1}(G_0)
e^{-nh}  < 1 -e^{-h}.
\end{equation} 
This can be done in many ways; the following procedure is perhaps slightly
brutal, but it is easy to describe. In the following a \emph{simple
  path} from $v$ to $w$ in a graph $K$ is a finite path $\mu = e_1e_2 \cdots e_n \in E^n$ in $P_f(K)$ such
that $s(e_1 ) = v, \ r(e_n) =w, \ s(e_j) \neq v \ \forall j \in
\{2,3,\cdots, n\}$. Let
$w_1,w_2,w_3,\cdots $ be a numbering of the vertexes in the attached
exit paths $t^i, \ i \geq 1$. Observe that for each $i$ there are only
finitely many, say $a_i$, simple paths
in $H$ from $v_1$ to $w_i$. Let $l_i$ be the maximal length of such a
path. Choose $n_1 \in \mathbb N$ so large that
$a_1 \leq e^{n_1\frac{h}{2}}$ and $\sum_{ j \geq n_1}
e^{-\frac{jh}{2}} < 1 - e^{-h}$,
and then recursively $n_{k+1} > n_k + l_k$ such that
$$
a_{k+1} \leq e^{\frac{n_{k+1}h}{2}} .
$$  
Add to $H$ for each $i$ a path $\mu_i$ from $w_i$ to $v_1$ of length $n_i$. Since every simple loop in $G_0$ from
$v_1$ back to itself must
pass through a $w_i$, we find that $l^n_{v_1v_1}(G_0) = 0$ unless $n
\in \bigcup_k [n_k,n_k+l_k]$ and hence
\begin{equation}
\begin{split}
&\sum_{n=1}^{\infty} l^n_{v_1v_1}(G_0)e^{-nh} = \sum_{k=1}^{\infty}
\sum_{j=n_k}^{n_k+l_k} l^j_{v_1v_1}(G_0)e^{-jh}   \leq
\sum_{k=1}^{\infty} \sum_{j=n_k}^{n_k+l_k} l^j_{v_1v_1}(G_0)e^{-n_kh}
\\
&  =
\sum_{k=1}^{\infty} a_k e^{-n_kh}  \leq \sum_{k=1}^{\infty}
e^{-\frac{n_kh}{2}} \leq \sum_{ j \geq n_1}
e^{-\frac{jh}{2}}  < 1 -e^{-h},
\end{split}
\end{equation}
as desired. To make the final adjustment to the graph we first use
Lemma \ref{djurs10} to get a sequence $\{b_n\}_{n=1}^{\infty}$ of non-negative
integers such that $b_1 = 1$, $b_n \geq l^n_{v_1v_1}(G_0)$ for all $n$,
$\limsup_n (b_n)^{\frac{1}{n}} = e^h$ and
$$
\sum_{n=1}^{\infty} b_n e^{-nh} = 1.
$$ 
For each $n \in \mathbb N$ we add $b_n-l^n_{v_1v_1}(G_0)$ arrows to $G_0$
going from $v_n$ to $v_1$. The resulting graph $G$ is then a recurrent Ruette
graph with Gurevics entropy $h(G) = h$ by Lemma \ref{ruette3}. By
Proposition \ref{ruette4} we obtain a transient graph $G'$ with the
same entropy by removing the single edge going from $v_1$ to
$v_1$. Both graphs are row-finite. We
claim $G$ will have the properties described in
Theorem \ref{REV1} and $G'$ the properties described in Theorem
\ref{REV2}. For both graphs it is easy to see that the exits
correspond bijectively to the exit paths $t^i, \ i \geq 0$. In view of
Theorem \ref{n11} it suffices therefore to show that each exit
path $t^i$ will be $\beta$-summable iff $\beta \in I_i$ and
$e^{-\beta}A$ is transient, where $A$ is the adjacency matrix of the
graph in question, i.e. either $G$ or $G'$. The argument for this is a
matter of book-keeping and runs as follows.

For the bare exit the claim follows from Lemma \ref{r11} so we
consider one of the attached exit paths, say $ t = t^M$ so that the graph
(\ref{G1}) depicts the relevant part of $G$.  For $k \geq 2$, let $\mathbb L_k(n)$ denote the set of simple paths
from $t_1$ to $t_k$ of length $n$, and
set 
$$
\mathbb L_k = \bigcup_n \mathbb L_k(n) .
$$
Let $\mathbb A_k$ be the set of all finite paths $\mu
\in P_f(G)$ such that $s(\mu) = t_1, \ r(\mu) = t_k$. Finally, we let
$\mathbb B$ be the set of all finite paths (loops) $\mu \in P_f(G)$ such that
$s(\mu) = r(\mu) = t_1$. Then
$$
\mathbb A_k = \sqcup_{\mu' \in \mathbb L_k} \left\{ \mu''\mu' : \
  \mu'' \in \mathbb B \right\} .
$$
It follows that
$$
\sum_{n=0}^{\infty} A^n_{t_1t_k} e^{-n \beta} =  \sum_{\mu \in \mathbb
  A_k} e^{-|\mu|\beta} = \alpha \sum_{\mu \in \mathbb L_k}
e^{-|\mu|\beta} 
$$
where
$$
\alpha  \defeq \sum_{n=0}^{\infty} A^n_{t_1t_1} e^{-n\beta}.
$$
Set 
\begin{equation}\label{X}
x_k =   e^{(k-1)\beta}
t(k)^{-1}\sum_{n=1}^{k-1} e^{-n\beta} \# \mathbb L_k(n) 
\end{equation}
and
\begin{equation}\label{Y}
y_k =  e^{(k-1)\beta}
t(k)^{-1}\sum_{n \geq k}e^{-n\beta} \# \mathbb L_k(n) .
\end{equation}
Then
$$
 e^{(k-1)\beta}
t(k)^{-1} \sum_{n=0}^{\infty} A^n_{t_1t_k}e^{-n\beta} = \alpha(x_k +
y_k) ,
$$
and $t$ will be a $\beta$-summable exit path iff $\alpha < \infty$ and
$\sup_k 
x_k$ and  $\sup_k 
y_k$ are both finite, or alternatively, since $x_k+y_k$ increases with
$k$, iff  $\alpha < \infty$ and $\sup_k 
x_{2k+1}$ and  $\sup_k 
y_{2k+1}$ are both finite. A simple count shows that
 \begin{equation}\label{djur1}
\sum_{n=1}^{2k} e^{-n\beta} \# \mathbb L_{2k+1}(n) = e^{-2k\beta} a_1a_2 \cdots a_{2k} + e^{-k\beta}d_k +
  \sum_{i=1}^{k-1} e^{- (2k-i)\beta} d_i a_{2i+1}a_{2i+2}
  \cdots a_{2k}  ,  
\end{equation}
for $k \geq 1$, 
while
\begin{equation}\label{djur2}
\sum_{n \geq k} e^{-n\beta} \# \mathbb L_k(n) 
= e^{-(2k-2)\beta} b_{k-1} + \sum_{i=1}^{k-2} e^{-(k+i-1)\beta}b_i a_{i+1}a_{i+2} \cdots a_{k-1}   ,
\end{equation}
for $k \geq 2$. Inserting (\ref{djur1}) and (\ref{djur2}) into (\ref{X}) and
(\ref{Y}), and remembering that $d_i = c_ia_{i+1} \cdots a_{2i}$, we find that 
\begin{equation*}
x_{2k+1}  =
   1+ 
  \sum_{i=1}^{k} e^{i\beta} c_i \left(a_{1}a_2 \cdots a_{i}\right)^{-1}  ,
\end{equation*}
for $k \geq 1$, and
\begin{equation*}
y_{k} =
  \sum_{i=1}^{k-1} e^{-i\beta}b_i \left(a_{1}a_2 \cdots a_{i}\right)^{-1}   ,  
\end{equation*}
for $k \geq 2$.
It follows that $\sup_k x_{2k+1}$ and $\sup_k y_{2k+1}$ will both be
finite if and only if $\beta \in J(\textbf{a},\textbf{b},\textbf{c}) =
I_M$.
\qed

\section{Graphs with infinite emitters} In this section we prove the
two Theorems \ref{intro1} and \ref{intro} from the introduction.

\subsection{The basic construction II:  Adding an exit}\label{constr2} 

Let $G_0$ be an arbitrary graph and $v$ a vertex in $G_0$. In analogy
with what we did in Section \ref{constr} we want to
add to $G_0$ a new exit, but this time the construction will make $v$
an infinite emitter, also when it wasn't in $G_0$. Let $I$ an arbitrary
interval in $]0,\infty[$ and choose with the aid of Lemma \ref{summint}
three sequences $\textbf{a} = (a_i)_{i=1}^{\infty},\textbf{b} = (b_i)_{i=1}^{\infty}
  , \textbf{c} = (c_i)_{i=1}^{\infty}$ of non-zero natural numbers
  such that $J(\textbf{a},\textbf{b},\textbf{c}) = I$. Set $t_1 = v$
  and add to $G_0$ the vertexes $t_2,t_3,t_4, \cdots$. Also, add $a_i$
  arrows from $t_i$ to $t_{i+1}, \ i \geq 1$, add
  $c_i$ arrows from $t_1$ to $t_{i+1}, \ i \geq 1$, and add a path of length $2i$
  from $t_1$ to $t_i, i\geq 2$, whose first edge has multiplicity
  $b_{i-1}$ and the rest multiplicity $1$. Then, if we ignore the part of $G_0$ with no contact to
  the new vertexes, the graph $G$ looks at follows.

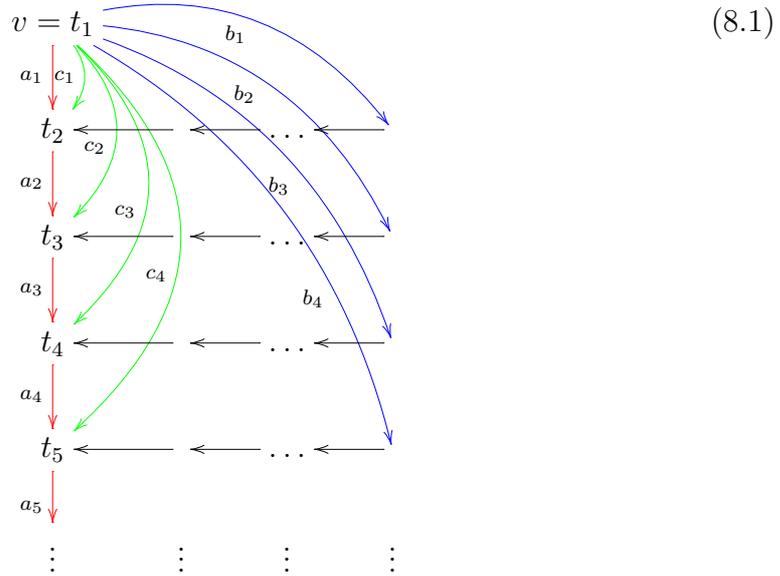
\begin{equation}\label{Grus}
\begin{xymatrix}{ v= t_1  \ar@[red][d]_{a_1} \ar@[green]@/^1pc/[d]_-{c_1}
    \ar@[green]@/^2pc/[dd]_(.60){   c_2} \ar@[green]@/^3pc/[ddd]_(.6){ c_3}  \ar@[green]@/^4pc/[dddd]_(.6){c_4}
    \ar@[blue]@/^2pc/[rrrd]_(.5){b_1}  \ar@[blue]@/^2pc/[rrrdd]_(.5){b_2} \ar@[blue]@/^2pc/[rrrddd]_(.6){b_3} \ar@[blue]@/^2pc/[rrrdddd]_(.7){b_4}& &  & \\ 
 t_2  \ar@[red][d]_{a_2}&
 \ar[l] & \hdots \ar[l] & \ar[l]\\
    t_3\ar@[red][d]_{a_3}
& \ar[l]&   \hdots \ar[l] &  \ar[l]  \\
 t_4 \ar@[red][d]_{a_4}
& \ar[l] & \ar[l]\hdots & \ar[l] \\
   t_5 \ar@[red][d]_{a_5} 
& \ar[l] & \ar[l] \hdots &\ar[l] \\
\vdots & \vdots & \vdots & \vdots } 
\end{xymatrix}
\end{equation}

The labels show the multiplicity of the edge; unlabeled black edges have
multiplicity 1.

\bigskip

\subsection{Proof of  Theorem \ref{intro1} and Theorem \ref{intro}:} The fundamental idea of the proof is the same as for the proof of
Theorems \ref{REV1} and \ref{REV2}: We construct a recurrent Ruette graph
$G$ with the properties described in Theorem \ref{intro1} and then
remove a single edge to obtain a graph with the properties described
in Theorem \ref{intro}. Let $M$ be a set which is countably infinite
if $N = \infty$ or else has $N-1$ elements, and let $I_1,I_2,I_3, \cdots$ be a numbering of the
  intervals in $\mathbb I$. Let $H_0$ be the graph
which only consists of the vertex $v$ and for each $i \in M$ a vertex
$w_i$, and an arrow from $v$ to $w_i$. Attach to $H_0$ an exit path
$t^1$ whose first vertex is $v$ as described in Section \ref{constr2} using multiplicities $\textbf{a(1)} =
\{a_i\},\textbf{b(1)} = \{b_i\}$ and $\textbf{c(1)} = \{c_i\}$ such that
$J(\textbf{a(1)},\textbf{b(1)},\textbf{c(1)}) = I_1$. Let $H_1$ be the
resulting graph. Attach to $H_1$ an exit path $t^2$ using multiplicities $\textbf{a(2)},\textbf{b(2)}$ and $\textbf{c(2)}$ such that
$J(\textbf{a(2)},\textbf{b(2)},\textbf{c(2)}) = I_2$. Let $H_2$ be the
resulting graph and continue recursively, either indefinitely or until
the intervals in $\mathbb I$ have been used up. Set
$$
H = \bigcup_i H_i .  
$$
We next add to $H$ a sequence $\mu_i , i =1,2,3, \cdots $, of paths going
from the
vertexes in the exit paths we have attached to the vertex $v$. The paths are mutually disjoint, both
from each other and from $H$, except for their initial and terminal
vertexes. We want also to arrange that in the resulting graph $G_0$ we
have that $l^1_{vv}(G_0) = 0$ and 
\begin{equation}\label{feber1}
\sum_{n=1}^{\infty} l^n_{vv}(G_0)
e^{-nh}  < 1 -e^{-h}.
\end{equation} 
This can be done in exactly the same way as in the proof in Section
\ref{constr}, but note that this time $G_0$ is still not strongly
connected unless $N=1$ since the vertexes $w_i,i \in M$, are sinks. To make the final adjustment to the graph we first use
Lemma \ref{djurs10} to get a sequence $\{b_n\}_{n=1}^{\infty}$ of non-negative
integers such that $b_1 = 1$, $b_n \geq l^n_{vv}(G_0)$ for all $n$,
$\limsup_n (b_n)^{\frac{1}{n}} = e^h$ and
$$
\sum_{n=1}^{\infty} b_n e^{-nh} = 1.
$$ 
By condition iii) in Lemma \ref{djurs10},
$$
W = \left\{ n \geq 2: \ b_n - l^n_{vv}(G_0) > 0 \right\}
$$
is infinite. When $N=1$ so that $M = \emptyset$, we add $b_n - l^n_{vv}(G)$ loops of
length $n$ from $v$ back to $v$ for each $n \in W$, all mutually disjoint except at
$v$. When $N \geq 2$ we partition $W$ into $\# M$ sets $W_i, i \in M$, with
infinitely many elements in each set. For each $k \in W_i$ we then add
$b_k-l^k_{vv}(G_0)$ paths of length $k-1$ from $w_i$ to $v$. The
paths must be mutually disjoint and disjoint from $G_0$ except at
their initial and terminal vertex. Finally, add a single edge from $v$
back to itself. In the resulting graph all $w_i, \ i \in M,$ are now
infinite emitters, and
$$
\sum_{n=1}^{\infty} l^n_{vv}(G) e^{-nh} = 1.
$$  
It follows from  Lemma \ref{ruette3} that $G$ is a recurrent Ruette
graph with Gurevics entropy $h(G) = h$. By
Proposition \ref{ruette4} we obtain a transient graph $G'$ with the
same entropy by removing the single edge going from $v$ to
$v$. To check that $G$ has the properties described in Theorem
\ref{intro1} and $G'$ the properties described in Theorem \ref{intro}
it remains only to check the assertions regarding the exits. For both graphs it is easy to see that the exits
correspond bijectively to the exit paths we have added. In view of
Theorem \ref{n11} it suffices therefore to show that the exit $t$
corresponding to one of the intervals in $\mathbb I$, say $I_i$, will be $\beta$-summable iff $\beta \in I_i$ and
$e^{-\beta}A$ is transient, where $A$ is the adjacency matrix of the
graph in question, i.e. either $G$ or $G'$. Let $\textbf{a} = (a_i)_{i=1}^{\infty},\textbf{b} = (b_i)_{i=1}^{\infty}
  , \textbf{c} = (c_i)_{i=1}^{\infty}$ be the three sequences of
  non-zero natural numbers used to construct $H_i$ from $H_{i-1}$ and
  having the property that $J(\textbf{a},\textbf{b},\textbf{c}) = I_i$. Using the
  same notation as in the proof of Theorems \ref{REV1} and \ref{REV2}
  we find this time that
\begin{equation*}
\begin{split}
& \sum_{n=1}^k e^{-n\beta} \# \mathbb L_{k+1}(n) = c_ke^{-\beta}  + a_kc_{k-1}e^{-2\beta} +
a_ka_{k-1}c_{k-2}e^{-3\beta} + \\
&\ \ \ \ \ \ \ \ \ \ \cdots 
+ a_ka_{k-1} \cdots a_2c_1e^{-k\beta} +a_ka_{k-1}a_{k-2}\cdots a_1
e^{-k\beta}  .
\end{split}
\end{equation*}
Hence 
\begin{equation*}
\begin{split}
 x_{k+1} =   e^{k\beta} t(k+1)^{-1} \sum_{n=1}^k e^{-n\beta} \#
 \mathbb L_{k+1}(n) = c_1^{-1}\left(
1 + e^{-\beta}\sum_{i=1}^k \frac{c_i}{a_1a_2\cdots a_i} e^{i\beta}\right) .\\
\end{split}
\end{equation*}
We see that the limit $\lim_{k \to \infty} x_k$ is finite if and only
if (\ref{m3}) holds. Similarly we find that
\begin{equation*}
\begin{split}
&  \sum_{n \geq k+1} e^{-n\beta} \# \mathbb L_{k+1}(n) = a_ka_{k-1} \cdots a_2 b_1e^{-(3+ k)\beta} +
a_ka_{k-1} \cdots a_3b_2e^{-(4 +k)\beta} + \\
& \ \ \ \cdots +
 a_kb_{k-1}e^{-(2k + 1)\beta} + b_ke^{-(2k+2) \beta}. \\
\end{split}
\end{equation*}
It follows that
\begin{equation*}
\begin{split}
&c_1 y_{k+1}  =  c_1e^{k\beta} t(k+1)^{-1}  \sum_{n \geq k+1} e^{-n\beta} \# \mathbb L_{k+1}(n)    \\
&= a_1^{-1} b_1e^{-3\beta} +
(a_2a_1)^{-1}b_2 e^{-4\beta} + \\
& \ \ \ \cdots +
 (a_{k-1}a_{k-2} \cdots a_1)^{-1}b_{k-1}e^{-(k +1)\beta} +
 (a_{k}a_{k-1} \cdots a_1)^{-1} b_{k}e^{-(k+2) \beta}. \\
\end{split}
\end{equation*}
Thus $\lim_{k \to \infty} y_k $ is finite if and only
 (\ref{m6}) holds. Since
$$
 e^{(k-1)\beta}
t(k)^{-1} \sum_{n=0}^{\infty} A^n_{vt_k}e^{-n\beta} = \alpha(x_k +
y_k) ,
$$
where $\alpha = \sum_{n=0}^{\infty} A^n_{vv} e^{-n \beta}$ we find
that $t$ is $\beta$-summable if and only if $e^{-\beta} A$ is
transient and $\beta \in J\left(\textbf{a},\textbf{b},\textbf{c}\right) = I_i$.

\qed

\section{Appendix}\label{appendix}

An exit path $(v_i)_{i=1}^{\infty}$ is
\emph{bare} when 
\begin{equation}\label{bare}
\#  s\left(r^{-1}\left(v_{i+1}\right)\right)  = 1
\end{equation} 
for all $i$, and \emph{eventually bare} when (\ref{bare}) holds for
all $i$ large enough. An exit is \emph{bare} when one of its
representing exit paths is bare, in which case they are all eventually
bare. The purpose with this Appendix is to prove

\begin{thm}\label{!!!} Let $G$ be a strongly connected row-finite graph with
  infinitely many vertexes and at most countably many exits. Then $G$
  contains a bare exit.
\end{thm}

The first step is to show that there exist exit paths.

\begin{lemma}\label{r4} Assume that $G$ is a strongly connected
  row-finite graph with infinitely many vertexes. Then $G$
  contains an exit path.
\end{lemma}
\begin{proof} This was proved by Van Cyr in his thesis, cf. page 94 in
  \cite{Cy}. Here is the argument: Let $v_0,v_1,v_2,v_3, \cdots$ be a
  numbering of the vertexes in $V$. For
  each $i$ choose a finite path $\mu_i$ from $v_0$ to $v_i$ of minimal
  length. Since $s^{-1}(v_0)$ is finite there are infinitely many
  $\mu_i$'s that share the first edge, $e_1$ say. Among them there are
  infinitely many that share the second edge, $e_2$, and so on. This
  results in an infinite path $p = e_1e_2e_3e_4\cdots $ in which the
  vertexes only occur once. It follows that the vertexes in $p$ form an exit path.  
\end{proof}

When $\mu = e_1e_2 \cdots e_n $ is a finite path in $G$ and $F \subseteq
V$ is a set of vertexes, we write $\mu \cap F = \emptyset$ when  $\mu$ does not contain any
vertex from  $F$, i.e.   
$$
\bigcup_{i=1}^n \left\{ s(e_i),r(e_i)\right\}  \cap F = \emptyset
.
$$
Let $(v_i)_{i=1}^{\infty}$ and $(w_i)_{i=1}^{\infty}$ be exit paths in
$G$. We write 
$$
(v_i)_{i=1}^{\infty} \ \leq \ (w_i)_{i=1}^{\infty}
$$
when the following holds: For every finite subset $F \subseteq V$ and
every $N \in \mathbb N$ there are $n,m \geq N$ and a finite path $\mu$
in $G$ such that $s(\mu) = v_n, \ r(\mu) = w_m$ and $\mu \cap F =
\emptyset$. Given two exits $t$ and $t'$ in $G$ we write $t\leq t'$
when there are exit paths, $(v_i)_{i=1}^{\infty}$ and
$(w_i)_{i=1}^{\infty}$, representing $t$ and $t'$, respectively, such that
$(v_i)_{i=1}^{\infty} \ \leq \ (w_i)_{i=1}^{\infty}$.

Let $\mathcal E$ be the collection of exits in $G$.

\begin{lemma}\label{!1}  Let $G$ be a strongly connected graph with
  infinitely many vertexes and at most countably many exits. $\mathcal
  E$ is partially ordered by the relation $\leq$.
\end{lemma}
\begin{proof} Only the anti-symmetry condition is not obvious. For
  this we must show that when $(v_i)_{i=1}^{\infty}$ and
  $(w_i)_{i=1}^{\infty}$ are exit paths such that
\begin{equation}\label{!2}  
(v_i)_{i=1}^{\infty} \ \leq \ (w_i)_{i=1}^{\infty} \ \leq \
  (v_i)_{i=1}^{\infty},
\end{equation} 
then the paths are tail-equivalent. 
It follows from (\ref{!2}) that we can go back and forth between
$(v_i)_{i=1}^{\infty}$ and $(w_i)_{i=1}^{\infty}$ with paths that
eventually leave all finite subsets of vertexes. In particular, we can find 
sequences $n_1 < n_2 < \cdots $ and $m_1 <
m_2 < \cdots $ in $\mathbb N$ and elements $\mu_i,\nu_i \in \bigcup_n V^n, i
=1,2,3, \cdots$, such that 
$$
z = v_{n_1}\mu_1w_{m_1}\nu_1v_{n_2}\mu_2 w_{m_2}\nu_2 \cdots
$$
is an exit path in $G$ with the property that $v_{n_i}$ does not occur
in $\mu_iw_{m_i}\nu_iv_{n_i+1}\mu_{i+1}w_{m_{i+1}} \cdots$ and $w_{m_i}$
does not occur in $\nu_iv_{n_{i+1}}\mu_{i+1}
w_{m_{i+1}}\nu_{i+1}\cdots $. For each $i
\in \mathbb N$, set $a_i^0 =  v_{n_i}v_{n_i+1} \cdots v_{n_{i+1} -1}$
and $a^1_i = v_{n_i}\mu_iw_{m_i}\nu_i$. For each
$(i_j)_{j=1}^{\infty} \in \{0,1\}^{\mathbb N}$ we can then consider the exit path
$$
P((i_j)_{j=1}^{\infty}) = a^{i_1}_1a^{i_2}_2a^{i_3}_3 a^{i_4}_4\cdots .
$$
If $a^0_i \neq a^1_i$ for infinitely may $i$ it follows that
$P\left(\{0,1\}^{\mathbb N}\right)$ contains uncountably many mutually
tail inequivalent exit paths, contradicting our assumption on $G$. It
follows therefore that there is an $i_0$ such that $a^0_i =
a^1_i$ for all $i \geq i_0$. This implies that $z$ is tail-equivalent
to $(v_i)_{i=1}^{\infty}$, and then by symmetry also to $(w_i)_{i=1}^{\infty}$.

\end{proof}

 Fix a vertex $v \in V$
  and an increasing sequence $F_1 \subseteq F_2 \subseteq F_3
  \subseteq \cdots$ of finite subsets of $V$ such that $v\in F_1$ and
  $\bigcup_n F_n = V$. Let $w =
  (w_i)_{i=1}^{\infty}$ be an exit path in $G$ with $w_1 = v$ or an
  element $w \in V^n$ with $w_1 = v$ and $w_n \notin F_k$. There is
  then a unique vertex
  $E(w,k)\in V$ and a unique $j\in \mathbb N$ such that
\begin{enumerate}
\item[i)] $w_j = E(w,k)$,
\item[ii)] $w_i \notin F_k, \ i \geq j$,
\item[iii)] $w_{j-1} \in F_k$.
\end{enumerate}
In words, $E(w,k)$ is the first vertex outside of $F_k$ which $w$ visits
when it leaves $F_k$ for good.

\begin{lemma}\label{!3}  Let $G$ be a strongly connected graph with
  infinitely many vertexes and at most countably many exits. Then $(\mathcal E,\leq)$ has a minimal element.  

\end{lemma}
\begin{proof} This follows from Zorn's lemma if we prove that a  maximal totally ordered subset $M$ in $(\mathcal
  E,\leq)$ has a minimal element. This is trivial if $M$ is finite. By
  assumption $M$ countable, say $M = \{t^i:  \ i = 1,2,3,
  \cdots \}$. For each $n \in \mathbb N$ we let $s^n$ be the minimal
  element in $\{t^i: \ i  = 1,2, \cdots, n\}$.
 Let $v^n =
  \left(v^n_i\right)_{i=1}^{\infty}$ be an exit path representing
  $s^n$ such that $v^n_1 = v$. Since $G$ is row-finite there are infinite subsets $N_1 \supseteq N_2
\supseteq N_3 \supseteq \cdots $ of natural numbers such that $E(v^n,k) = E(v^{n'},k)$
for all $n,n' \in N_k$. We can therefore piece together an exit path
$u = (u_i)_{i=1}^{\infty}$ such that for all $k \in \mathbb N$ there
are infinitely many $n \in \mathbb N$ with $E(v^n,k) =E(u,k)$. It is
easy to see that $ (u_i)_{i=1}^{\infty} \leq  \left(v^m_i\right)_{i=1}^{\infty}$ for all $m \in \mathbb N$. It follows
that $u$ represents an exit $[u]$ such that $[u] \leq t$ for all $t
\in M$.  
\end{proof}


We can now finish the proof of Theorem \ref{!!!} with the following

\begin{lemma}\label{!5} Let $G$ be a strongly connected graph with
  infinitely many vertexes and at most countably many
  exits. A minimal element in $(\mathcal E,\leq)$ is represented by a
  bare exit path.
\end{lemma} 
\begin{proof} Let $u$ be an exit path such that the exit $[u]$ it represents
  is minimal in $(\mathcal E,\leq)$. Assume for a contradiction that
\begin{equation}\label{!6}
\#s\left(r^{-1}(u_{i+1})\right) \geq 2
\end{equation}
for infinitely many $i$. There are then also infinite many $i \in
\mathbb N$ such that (\ref{!6}) holds, and at the same time
\begin{equation}\label{!7}
k > i+1 \ \Rightarrow \ u_k \neq u_{i+1} .
\end{equation}
 For each $i$ for which (\ref{!6}) and (\ref{!7}) both hold we choose
an edge $x_i \in r^{-1}(u_{i+1})$ with $s(x_i) \neq u_i$. Note that
$\lim_{i \to \infty} s(x_i) = \infty$ since $G$ is row-finite. Without
loss of generality we can therefore assume that $s(x_i) \notin
F_i$. For each $i$ we choose a finite path $\mu_i$ in $G$ with $s(\mu_i) = v$
such that $x_i$ is the last edge in $\mu_i$. Then $E(\mu_i,k)$ is
defined for all $i \geq k$ and since $G$ is row-finite there is a sequence $N_1
\supseteq N_2 \supseteq N_3 \supseteq \cdots$ of infinite subsets in
$\mathbb N$ such that $E(\mu_i,k)$ is the same vertex for all $i \in N_k$.
Let $w_k$ be this common vertex. We
can then piece together an exit path $y = (y_i)_{i=1}^{\infty}$ such
that there is a sequence $l_1 \leq l_2 \leq \cdots$ in $\mathbb N$ with
$y_{l_k} = w_k$ for all $k \in \mathbb N$. By construction there is
then also a sequence $j_1 <
j_2 < \cdots$ in $\mathbb N$ and for
each $i \in \mathbb N$ a finite path $\mu'_i$ in $G$ such that $s(\mu'_i) =
y_{l_i}$, $\mu'_i \cap F_i = \emptyset$ and such that the last edge in $\mu'_i$
is $x_{j_i}$. Since the set $N_i$ is infinite we can choose $j_i$
arbitrarily large, for each $l_i$. The paths $\{\mu'_i\}$ ensure that $y \leq u$ and the minimality of $[u]$
implies then that $y$ and $u$ are tail-equivalent.

It follows that $u$ has the
following property: For each $i \in \mathbb N$ there is a finite path
$\nu_i$ in $G$ and natural numbers $l_i,k_i \in \mathbb N$ such that $\nu_i \cap F_i =
\emptyset$, $l_i < k_i$, $s(\nu_i) = u_{l_i}$, $r(\nu_i) = u_{k_i}$,
$j > k_i \Rightarrow u_j \neq u_{k_i}$, and the last edge in
$\nu_i$ is not in $s^{-1}\left(u_{k_i-1}\right)$. Furthermore,
choosing the $\nu_i$ recursively we can
arrange that $k_i < l_{i+1}$ for all $i$ and that $\nu_{i+1}$ does not
contain a vertex which occurs in $\nu_j$ for some $j \leq i$. Set then $a^0_i = u_{l_i}u_{l_i+1} \cdots u_{k_i}$ and let
$a^1_i \in \bigcup_n V^n$ be the string of vertexes in $\nu_i$, starting with $u_{l_i}$ and ending with
$u_{k_i}$. Define $P :\{0,1\}^{\mathbb N} \to P(V)$ such that
\begin{equation*}
\begin{split}
&P((i_j)_{j=1}^{\infty}) \\
&= a^{i_1}_1
u_{k_1+1}u_{k_1+2} \cdots u_{l_2-1}a^{i_2}_2u_{k_2+1}u_{k_2+2} \cdots
u_{l_3-1} a^{i_3}_3 u_{k_3+1}u_{k_3+2} \cdots
u_{l_4-1}a^{i_4}_4\cdots .
\end{split}
\end{equation*}
Note that the last occurrence in $P((i_j)_{j=1}^{\infty})$ of each $u_{k_i}$ and the vertex
preceding it give away $(i_j)_{j=1}^{\infty}$. We conclude therefore
that $P$ is injective and that $P\left(\{0,1\}^{\mathbb N}\right)$
contains uncountably many tail-equivalence classes of exit paths in
$G$, contradicting the assumption. This contradiction shows that $u$
must be eventually bare.

\end{proof}


\begin{thebibliography}{WWWWW} 








\bibitem[BR]{BR} O. Bratteli and D.W. Robinson, {\em Operator Algebras
    and Quantum Statistical Mechanics I + II}, Texts and Monographs in
  Physics, Springer Verlag, New York, Heidelberg, Berlin, 1979 and 1981.


\bibitem[BEK]{BEK} O. Bratteli, G. Elliott and A. Kishimoto, {\em The
    temperature state space of a dynamical system I}, J. Yokohama
  Univ. {\bf 28} (1980), 125-167. 


\bibitem[CL]{CL} T.M. Carlsen and N. Larsen, {\em Partial actions and
    KMS states on relative graph $C^*$-algebras}, Preprint, Nov. 2013. 



\bibitem[CT1]{CT1} J. Christensen and K. Thomsen, {\em Finite digraphs
    and KMS states}, J. Math. Anal. Appl. {\bf 433} (2016), 1626-1646.w



\bibitem[CT2]{CT2}  J. Christensen and K. Thomsen, {\em Diagonality of
    actions and KMS weights}, preprint, arXiv:1512.05086.





\bibitem[Co]{Co} D.L. Cohn, {\em Measure Theory}, Birkh\"auser,
  Boston, Basel, Berlin, 1997.



\bibitem[C]{C} F. Combes, {\em Poids associ\'e \`a une alg\`ebre
    hilbertienne \`a gauche}, Compos. Math. {\em 23} (1971), 49-77.

  
 
 
 
\bibitem[Cy]{Cy} V.T. Cyr, {\em Transient Markov Shifts},
  Ph.D. thesis, Pennsylvania State University, August 2010.








\bibitem[EFW]{EFW} M. Enomoto, M. Fujii and Y. Watatani, {\em KMS
    states for gauge action on $O_A$}, Math. Japon. {\bf 29} (1984), 607-619.




\bibitem[EL]{EL} R. Exel and M. Laca, {\em Partial dynamical systems
    and the KMS condition}, Comm. Math. Phys. {\bf 232} (2003), 223-277.




\bibitem[Ki]{Ki} A. Kishimoto, {\em Locally representable
    one-parameter automorphism groups of AF algebras},
  Rep. Math. Phys. {\bf 45} (2000), 333-356.































\bibitem[KV]{KV} J. Kustermans and S. Vaes, {\em Locally compact
    quantum groups}, Ann. Scient. \'Ec. Norm. Sup. {\bf 33}, 2000, 837-934.

 \bibitem[LN]{LN} M. Laca and S. Neshveyev, {\em KMS states of
     quasi-free dynamics on Pimsner algebras}, J. Func. Analysis {\bf
     211} (2004), 457-482.



\bibitem[N]{N} S. Neshveyev, {\em KMS states on the $C^*$-algebras of
    non-principal groupoids}, J. Operator Theory {\bf 70} (2013), 513-530.







\bibitem[Pa]{Pa} A.L.T. Paterson, {\em Graph inverse semigroups,
    groupoids and their $C^*$-algebras}, J. Operator Theory {\bf 48}
  (2002), 645-662.



 
\bibitem[Ph]{Ph} N. C. Phillips, {\em Crossed Products of the Cantor Set by Free Minimal Actions of $\mathbb Z^d$}, Comm. Math. Phys. {\bf 256} (2005), 1--42.



\bibitem[Pr]{Pr} W.E. Pruitt, {\em Eigenvalues of non-negative
    matrices}, Ann. Math. Statist. {\bf 35} (1964), 1797-1800.




\bibitem[Re]{Re} J. Renault, {\em A Groupoid Approach to $C^*$-algebras},  LNM 793, Springer Verlag, Berlin, Heidelberg, New York, 1980.  



\bibitem[Ru]{Ru} S. Ruette, {\em On the Vere-Jones classification and
    existence of maximal measure for countable topological Markov
    chains}, Pac. J. Math. {\bf 209} (2003), 365-380.

\bibitem[S]{S} O. Sarig, {\em Thermodynamic formalism for
    countable Markov shifts}, Ergodic Th. \& Dynam. Syst. {\bf 19}
  (1999), 1565-1593.





\bibitem[Sa]{Sa} S.A. Sawyer, {\em Martin boundary and random walks},
  Harmonic functions on trees and buildings (New York, 1995), 17-44,
  Contemp. Math. 206, Amer. Math. Soc., Providence, RI, 1997.
 


 


\bibitem[Sz]{Sz} W. Szymanski, {\em Simplicity of Cuntz-Krieger
    algebras of infinite matrices}, Pac. J. Math. {\bf 122} (2001), 249-256.





\bibitem[Th1]{Th1} K. Thomsen, {\em KMS weights on groupoid and graph
    $C^*$-algebras}, J. Func. Analysis {\bf 266} (2014), 2959-2988.


\bibitem[Th2]{Th2} K. Thomsen, {\em Dissipative conformal measures on
    locally compact spaces},  Ergod. Th. \& Dynam. Syst. {\bf 36}
  (2016), 649-670.


\bibitem[Th3]{Th3} K. Thomsen, {\em On the positive eigenvalues and
    eigenvectors of a non-negative matrix}, to appear in the Abel
  Symposium 2015 Proceedings. arXiv:1306.5116.




\bibitem[V]{V} D. Vere-Jones, {\em Ergodic properties of
    non-negative matrices I}, Pacific J. Math. {\bf 22} (1967), 361-386.


\bibitem[Wo]{Wo} W. Woess, {\em Denumerable Markov Chains}, EMS
  Textbooks in Mathematics, 2009.


\end{thebibliography}
\end{document}